\documentclass[11pt]{article}
\usepackage{amsmath}
\usepackage{amsthm}
\usepackage{amsfonts}
\usepackage{latexsym}
\usepackage{graphicx}
\usepackage{caption}
\usepackage[pdftex,bookmarks=true,bookmarksnumbered=true]{hyperref}
\usepackage{color}
\renewcommand{\epsilon}{\varepsilon}
\renewcommand{\rho}{\varrho}
\renewcommand{\phi}{\varphi}

\newcommand{\R}{{\mathbb R}}
\newcommand{\C}{{\mathbb C}}

\newcommand{\cB}{{\cal B}}

\newcommand{\cG}{{\cal G}}

\newcommand{\cL}{{\cal L}}

\newcommand{\cP}{{\cal P}}

\newcommand{\cX}{{\cal X}}
\newcommand{\cY}{{\cal Y}}
\begin{document}
\newtheorem{definition}{Definition}[section]
\newtheorem{theorem}[definition]{Theorem}
\newtheorem{proposition}[definition]{Proposition}
\newtheorem{corollary}[definition]{Corollary}
\newtheorem{lemma}[definition]{Lemma}
\newtheorem{remark}[definition]{Remark}
\newtheorem{conjecture}[definition]{Conjecture}
\newtheorem{problem}[definition]{Problem}
\title{A Computer-Assisted Study of Red Coral Population Dynamics}
\author{{\em Sayomi Kamimoto$^1$},
        {\em Hye Kyung Kim$^2$},
        {\em Evelyn Sander$^1$},
        {\em Thomas Wanner$^1$} \\[2ex]
        1. Department of Mathematical Sciences, George Mason University, \\
           Fairfax, VA 22030, USA\\
        2. School of Mathematics, University of Minnesota, \\
           Minneapolis, MN 55455, USA   
    }
%
\date{September 30, 2020}
\maketitle
%
%
\begin{abstract} 

We consider a 13-dimensional age-structured discrete red coral population model
varying with respect to a fitness parameter.  Our numerical 
results give a bifurcation diagram of both equilibria and stable invariant curves of orbits. 
We observe that not only for low levels of fitness, but also for high levels of fitness, 
populations are extremely vulnerable, in that they spend long time periods near extinction. 
We then use computer-assisted proofs techniques to rigorously validate the set of regular and 
bifurcation fixed points that have been found numerically. 

  
\end{abstract}
%
%
\setcounter{tocdepth}{1}
\tableofcontents
%
%
%
\section{Introduction}
\label{sec:intro}
%
%
%
%
Coral plays an important role in the marine ecosystem, and coral reefs
provide habitats to many sea animals and protect coastlines from
breaking waves and storms. Red coral is a long-lived,
slow-growing species, dwelling on Mediterranean rocky bottoms. Red coral populations are at risk due to both
global climate change and overharvesting~\cite{bramanti:etal:05a}.  
Bramanti, Iannelli, and Santangelo~\cite{bramanti:etal:09a,santangelo:etal:07a} investigated red coral
populations by scraping samples  off the coast of Italy in Calafuria
in the Western Ligurian Sea ($43^\circ 30^\prime$ N, $10^\circ 20^\prime$ E, Italy,
at a depth between 20 and 45m depth)  and observing
their growth rate over a four-year period. They used this data to
construct a Leslie-Lewis transition matrix, a static life table, and a 13-dimensional 
dynamical population model. Using this model, they studied population
trends by comparing small young colonies and bigger older colonies.
However, they only considered a small range of population trends.
In the current paper, we present a systematic
study of this  coral population model, shedding light on the long-term
dynamics of the red coral populations. We can  see the long-term
effect of change in reproduction fitness. We establish the equilibrium
structure and bifurcation points for the model, find a set of stable periodic invariant
cycles,  and show that for a large range of reproduction fitness these cycles 
get close to population extinction. 

In addition to these observations, we present and implement methods which 
allow us to rigorously validate the model's equilibrium and
bifurcation structure, including  both a saddle-node and a Neimark-Sacker
bifurcation. These validations  use a modification of the Newton-Kantorovitch type 
method developed in~\cite{sander:wanner:16a, wanner:17a, wanner:18a}.
While the previous version of this method merely used natural continuation,
this paper contains an extension of these results in which we 
consider rigorous validation using pseudo-arclength continuation~\cite{govaerts:00, keller:87}. In addition,
we use computer-assisted proof methods to prove the existence of saddle-node
and Neimark-Sacker bifurcation points on the equilibrium branch. These methods
significantly extend the range of applications of the constructive implicit
function theorem which was introduced in~\cite{sander:wanner:16a}. While for
the purposes of this paper we restrict ourselves to the case of finite-dimensional
Euclidean spaces, the results can easily be adapted to the general Banach space
setting, with little change. Thus, the pseudo-arclength results can be used 
for example in the setting of partial differenial equations, such as
the setting described in~\cite{sander:wanner:20a}. In other words, the present 
paper presents a functional analytic foundation for using pseudo-arclength
continuation in the context of computer-assisted proofs based on the 
constructive implicit function theorem presented in~\cite{sander:wanner:16a}.

The remainder of this paper is organized as follows.
We introduce the age-based red coral model in Section~\ref{sec:model}. In addition,
we present a bifurcation diagram of fixed points and stability of
the model, along with a detailed discussion of oscillations. These results show
how even at high fitness levels, the oscillations lead to extreme vulnerability of
the population. 
Section~\ref{sec:arclength} contains a functional-analytic approach to the rigorous
validation of the regular branches in the bifurcation diagram, which is based on a
constructive version of the implicit function theorem. Subsequently,
Section~\ref{sec:validation} details the validation for the three
bifurcation points on the main fixed point branch; namely,  the saddle-node
bifurcation in~\ref{subsec:saddle}, the Neimark-Sacker bifurcation
in~\ref{subsec:hopf}, and the transcritical bifurction in~\ref{subsec:transcritical}.
Section~\ref{sec:con} contains conclusions and future work.
\begin{figure}[tb] \centering
  \setlength{\unitlength}{1 cm}
    \includegraphics[height=5cm]{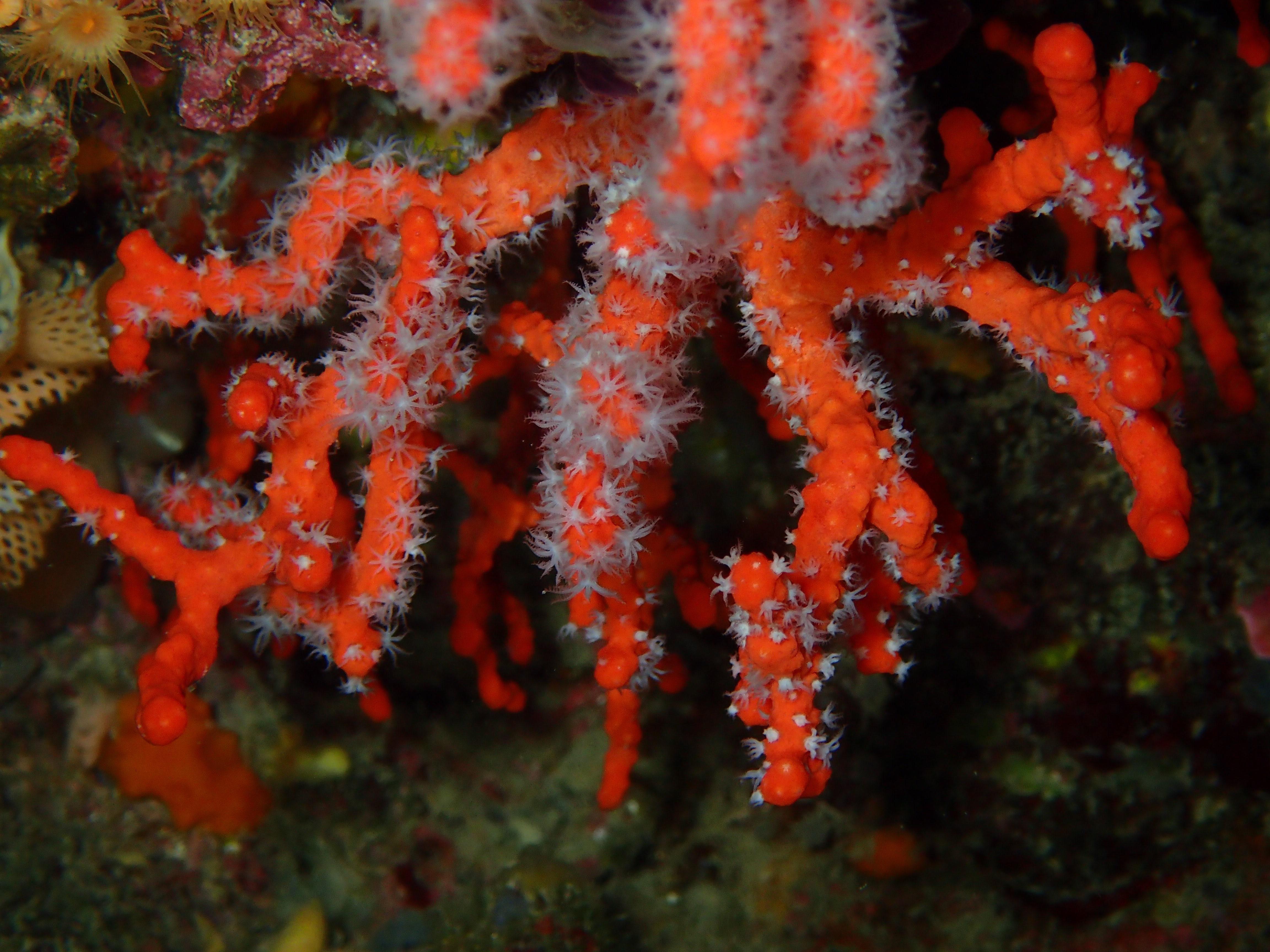}
    \hspace{0.5cm}
    \includegraphics[height=5cm]{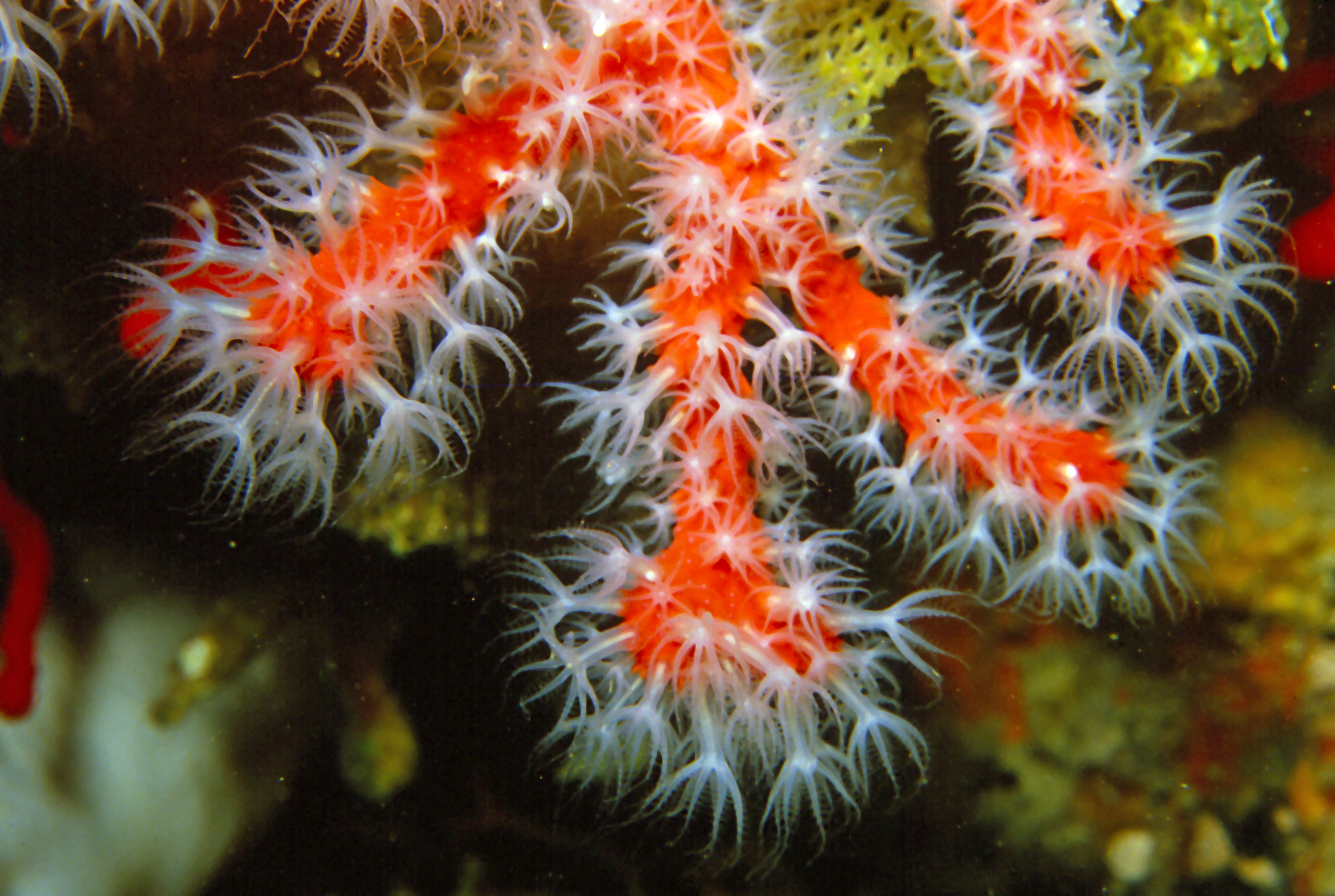}
  \caption{Photographs of red coral colonies. The individual polyps are visible particularly in the right-hand image. 
  Photos from ~\cite{coral1, coral2}.}
  \label{fig:redcoral}
\end{figure}%

\section{Red coral population model}
\label{sec:model}

In this section we present the red coral population model of Bramanti, Iannelli,
and Santangelo~\cite{bramanti:etal:09a,santangelo:etal:07a}, based on their experimental
and field data and a Leslie-Lewis transition matrix. In addition, we describe
the dynamics of the model in terms of its bifurcation structure and discuss
its implications.

\subsection{Description of the model}

A coral population is a self-seeding independent group consisting of
{\em polyps}, tiny soft-bodied organisms related to jellyfish. Polyps
form into {\em colonies}, which are distinct clusters
with polyps residing on a surface, as shown in Figure~\ref{fig:redcoral}.
A polyp is born to a parent colony in a
free-swimming larval stage. At the end of the larval stage, the polyp
permanently attaches itself to a colony and cannot move again. The age
of a colony has implications  in terms of its size and polyp density. As
a result, colony age determines the polyp attachment rate, the larval birth
rate, and the polyp survival rate. Based on these factors, larvae
will attach either to an existing colony or, especially if there
is a high polyp density, {\em recruitment} will occur. That is, 
larvae do not attach to existing colonies, but instead form new colonies. 
Red coral polyps can reproduce larvae starting  two years
after their birth, implying that there is no birth in a colony less than two 
years old, since none of the polyps are old enough to reproduce. 
Reproduction occurs at a discrete time in summer,
implying that a discrete population model is a natural modeling
assumption. 
\begin{figure}
\centering
\begin{minipage}{0.43\textwidth}
\centering
\includegraphics[width=6.6cm]{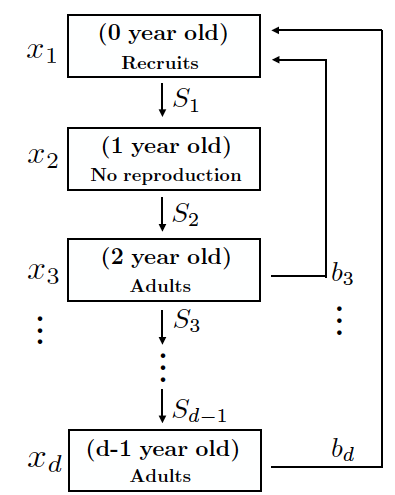}
\caption{Life cycle of coral population}\label{fig:sketch}
\end{minipage}
\hspace*{0.5cm}
\begin{minipage}{0.52\textwidth}
\centering
\captionsetup{type=table} 
\begin{tabular}{| c | c | c | c|}
    \hline
   Class $k$ &  Survival rate $S_k$ & Fertility $F_k$ \\
    \hline
    1 &  0.89 & 0  \\
    2 &  0.63 & 0  \\
    3 &  0.70 & 0.36  \\ 
    4 &  0.52 & 0.64  \\ 
    5 &  0.44 & 0.82 \\
    6 &  0.29 & 0.97 \\
    7 &  0.57 & 0.98 \\
    8 &  0.33 & 0.99  \\
    9 &  0.75 &   1 \\
    10 & 1    &   1  \\
    11 & 0.33 &   1  \\
    12 & 1    &   1  \\
    13 &      &   1  \\
    \hline
  \end{tabular}
\caption{Observational red coral data from~\cite{santangelo:etal:07a}.
  Our calculations are based on their fitting functions given in~(\ref{eqn:pkbk}) 
  and~(\ref{eqn:phi}), which were established using this data.}\label{tab:data}
\end{minipage}
\end{figure}

Based on the setting above, rather than modeling the total large number
of polyps in a coral population, the age-based model is a discrete time
model for $(x_1, x_2, \dots, x_d)$, where~$x_k$ is the number of colonies
of age group~$k$. The value~$d$ is the oldest colony in the population.
While in principle this~$d$ could be large, in the observations made
there was no colony of age group greater than~13. The value of~$x_k$
changes with respect to time (in years), where~$x_k^n$ denotes the
number of colonies of age group~$k$ at year~$n$. The colony life cycle
is displayed in the schematic diagram shown in Figure~\ref{fig:sketch}.
The downward arrows in Figure~\ref{fig:sketch} indicate that~$x_{k+1}$,
the number of colonies in age group~$k+1$, is determined exclusively
by the number of colonies in age group~$k$ in the previous year. This
relation is linear with respect to population, with the survival rate
constant~$S_k$. That is, we have $x_{k+1}^n = S_k x_{k}^{n-1}$. 
The survival rate values are determined by observation, and are  
given in Table~\ref{tab:data}, based on~\cite[Table~2]{santangelo:etal:07a}.

The upward arrows Figure~\ref{fig:sketch} indicate that recruits may be larvae
from any colony of age two or greater. Though it is not obvious from the
schematic diagram, the recruitment rate is not linear, and it depends on
both the total number of polyps in the colonies, as well as on the larvae
birth rates. Considering that the base variables~$x_k$ denote the number
of colonies in age group~$k$, the total number of polyps can be deduced
from the numbers~$p_k$ of polyps per colony in a colony of age group~$k$,
and the birth rates~$b_k$ depend on the fertility rates~$F_k$ given in
Table~\ref{tab:data}. Combined with the observational data
in~\cite{bramanti:etal:09a}, Bramanti et al.\ have then derived empirical
expressions for the polyp per colony numbers~$p_k$ and the birth
rates~$b_k$, which are given by
\begin{equation}\label{eqn:pkbk}
  p_k = 1.239 \, k^{2.324}
  \qquad\mbox{ and }\qquad
  b_k = F_k \, k^{2.324} \; . 
\end{equation}
For our calculations in the present paper, we use these fitting functions
rather than the original data, in keeping with the equations in~\cite{bramanti:etal:09a}. 
In addition to the birth rates, the number of recruits~$x_1$ depends also 
on a nonlinear function~$\phi$, which in turn depends on the density of polyps
per unit area. This function~$\phi$ is given by
\begin{equation}\label{eqn:phi}
  \phi(y) = \dfrac{c_1 e^{-\alpha y}}{y^2 + c_2 e^{-\beta y}}, \; 
  \quad\mbox{ with }\quad
  c_1 = 1.8 \cdot 10^5, \;
  c_2 = 1.3 \cdot 10^7, \;
  \alpha = 5 \cdot 10^{-4}, \;
  \beta = 3.4 \cdot 10^{-3} \, , 
\end{equation}
which again is a fit for the observational data in~\cite{bramanti:etal:09a}. The shape
of this nonlinearity is depicted in Figure~\ref{fig:lsr}. For a small density of
polyps, the function~$\phi$ increases with polyp density, whereas too large of a
polyp density inhibits the creation of new colonies due to competition for
resources.
\begin{figure}[tb] \centering
  \setlength{\unitlength}{1cm}
      \includegraphics[width=10cm]{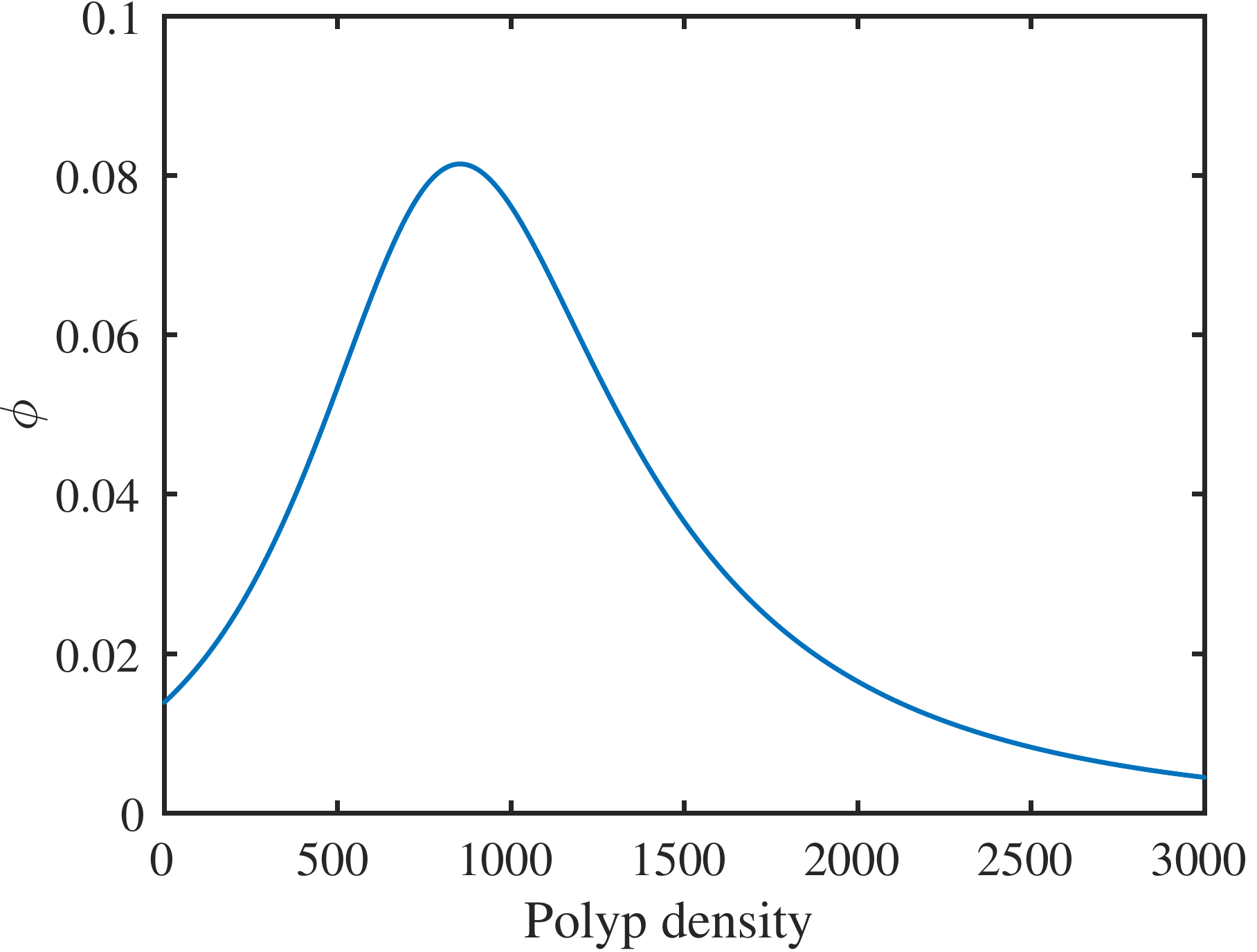}
  \caption{The recruits-to-larvae ratio function~$\phi$ plotted with respect
           to polyp density~$P$.  }
  \label{fig:lsr}
\end{figure}%

We now explain how to compute the polyp population density~$P$. We have already 
seen that the numbers~$p_k$ of polyps per colony in a colony of age group~$k$
satisfy the empirical formulas in~(\ref{eqn:pkbk}). Thus, the total number of
polyps in age group~$k$ is given by~$p_k x_k$. Now let~$\Omega$ denote the total
area of the population site, which was measured to be equal to 36 dm$^2$
in~\cite{bramanti:etal:09a}. Moreover, let $x = (x_1, x_2, \dots, x_d)$ be a column
vector giving the number of colonies of each age group, and let $p = (p_1,p_2,
\dots, p_d)$ denote the vector of polyps per colony in each age group. Then the
total number of polyps in the (non-recruit!) population~$Q$ and the polyp population
density~$P$ satisfy the identities
\begin{equation}\label{eqn:polypdensity}
 Q = \sum_{k=2}^d p_k x_k
 \qquad\mbox{ and }\qquad
 P = \frac{Q}{\Omega} \; . 
\end{equation}
Based on these preliminaries, let $x^n = (x_1^n,x_2^n, \dots, x_d^n)$
represent the vector containing the number of colonies at year~$n$, and
let~$P$ be the polyp population density defined in~(\ref{eqn:polypdensity}). 
If we now define
\begin{equation} \label{eqn:model1}
L(\lambda,x) = 
\begin{bmatrix}
    \lambda b_1 \phi(P) & \lambda b_2 \phi(P) & \lambda b_3 \phi(P) & \dots
      & \lambda b_{d-1} \phi(P) & \lambda b_d \phi(P)\\
    S_1 & 0 & 0 & \dots & 0 & 0 \\
    0 & S_2 & 0 & \dots & 0 & 0 \\
    0 & 0 & S_3 & \dots & 0 & 0 \\
    \vdots & \vdots & \ddots & \vdots & \vdots\\
    0 & 0 & 0 & \dots & S_{d-1} & 0 \\
\end{bmatrix} \; ,
\end{equation}
where the bifurcation parameter~$\lambda$ is described below,
then our model is given by
\begin{equation} \label{eqn:model2}
  x^{n+1} = L(\lambda,x^n) x^n \; . 
\end{equation}
The model~(\ref{eqn:model1}) and~(\ref{eqn:model2}) is an age-structured,
nonlinear, discrete-time dynamical model. For the parameter value $\lambda = 1$,
it is precisely based on the observational data in~\cite{bramanti:etal:09a}.
The nonlinearity arises only in the evolution of the variable~$x_1$, which
describes the number of recruit colonies. In a slight reformatting of notation,
let the function $f: \R \times \R^d \to \R^d$ be given by $f(\lambda,x) = L(\lambda,x) x$.
Then $x^{n+1} = f(\lambda, x^n)$, meaning that the dynamical population variation
corresponds to the iteration of the parameter-dependent nonlinear map~$f$. 

We still have to justify the introduction of the bifurcation parameter~$\lambda$
in the above formulas. Previous work
concentrated on the effect of varying the biologically relevant reproductive
number~$R$, the total number of larvae produced by a single colony during its life span. 
This parameter is directly proportional to $\lambda$, as we will show in
Section~\ref{sec:basicprepro}.  The birth rate parameters~$b_k$ in the above
equation are determined by observation of a specific coral population over a
small time period. In order to consider a population model in which the population
is placed under stress, such as in the case of climate change, it is necessary
to change the parameters beyond what has been observed. While we could also consider
modification of other parameters, we choose to follow along the lines of~\cite{bramanti:etal:09a} 
and vary the birth rates, making the assumption that every birth rate parameter
will be equally affected. Therefore, in our subsequent analysis, for every~$k$
we let the birth rate be given by $\lambda b_k$, a fixed scaling factor compared
to the originally observed birth rate.
\subsection{Fixed points of the coral population model} 
\label{sec:fp}
We now consider the set of fixed points for the coral population model, given
by the nonlinear function~$f$ defined above, and how this set changes as a
function of the parameter~$\lambda$. That is, we wish to determine the set
of all pairs~$(\lambda,x) \in \R \times \R^d$ such that $f(\lambda,x) = x$.
As it turns out, this can be reformulated equivalently as a one-dimensional
problem. To see this, assume that we have $x = f(\lambda,x)$. Then for all
indices $k = 1,\ldots,d-1$ one has $x_{k+1} = S_k x_k$. Using these statements
iteratively, one readily obtains
\begin{displaymath}
    x_2 = S_1  x_1 \; , \quad
    x_3 = S_2 S_1 x_1 \; , \quad
	x_4 = S_3 S_2 S_1 x_1 \; , \quad
    \ldots \quad
    x_d = S_{d-1} \cdots S_2 S_1 x_1 \; . 
\end{displaymath}
Thus, for all $k = 2,\ldots,d$ we have $x_k = a_k x_1$, where one uses
the abbreviation
\begin{equation} \label{def:a}
   a_k = \prod_{i=1}^{k-1} S_i \; ,
\end{equation}
and we further define $a_1 = 1$ then one also has $x_1 = a_1 x_1$.
Since we can write each component~$x_k$ for $k \ge 2$ as a function of~$x_1$
alone, the fixed point problem is a one-dimensional problem, which is only
a matter of determining~$x_1$. Recall that we defined the polyp population
density~$P$ in~(\ref{eqn:polypdensity}), and let $b = (b_1,b_2, \dots, b_d)$. 
Then the equation for~$x_1$ is given by
\begin{displaymath}
  x_1 = \lambda (b \cdot x) \; \phi(P)  \; . 
\end{displaymath}
Moreover, let $a = (a_1, a_2, \dots, a_d)$. This immediately implies
the identities
\begin{displaymath}
  x = x_1 a \; , \qquad
  P = \frac{x_1}{\Omega} \sum_{k=2}^d p_k a_k \; ,
  \qquad\mbox{ and }\qquad
  b \cdot x = (b \cdot a) \; x_1 \; .
\end{displaymath}
Altogether, this shows that a vector~$x = (x_1,\ldots,x_d)$ is a fixed
point for the map~$f(\lambda,\cdot)$ if and only if $x = x_1 a$ and its
first component~$x_1$ satisfies the nonlinear equation
\begin{equation}\label{eqn:x1}
  x_1 = \lambda \, ( b \cdot a) \,  x_1 \,
    \phi \left( \frac{x_1}{\Omega} \sum_{k=2}^d p_k a_k \right) \; .
\end{equation}
From this equation, one can then determine all fixed
points of the coral population model. Notice that we clearly have the
trivial solution~$x = 0$ for all values of the parameter~$\lambda$, 
which corresponds to an extinct population.

\subsection{The basic reproduction number}\label{sec:basicprepro}

An important biological parameter for the coral population is the total
number of larvae produced by a single colony in its entire life span.
This number only depends on the birth and survival rates, and one can
easily see that it is given by
\begin{equation}
  R = \lambda b_1 + \lambda b_2 S_1 + \lambda b_3 S_2 S_1 + \dots +
    \lambda b_d S_{d-1} S_{d-2} \dots S_1 
    = \lambda \sum_{i=1}^d a_i b_i \; .
\end{equation}
The number~$R$ is called the {\em basic reproduction number\/}.
Using the notation from the last subsection, the above equation can be
rewritten as
\begin{equation}
  R = (b \cdot a) \lambda \; .
\end{equation}
In particular, while it is possible to vary~$R$ in such a way that the
relationship between the birth rate constants vary, under our assumptions,
the vectors~$b$ and~$a$ are fixed constant vectors, and we therefore have
a fixed linear relationship between~$R$ and~$\lambda$. To make it easy to
compare our results with those of previous papers, we have chosen to plot
all bifurcation diagrams with respect to the basic reproduction number~$R$.

\subsection{The fixed point bifurcation diagram}
\label{sec:diagram}

We now turn our attention to a description of the bifurcation diagram
of the fixed points for the coral population system. This diagram is
shown in Figure~\ref{fig:bif}, where the set of fixed points is plotted in
terms of the reproductive number~$R$ versus polyp population density~$P$.
The color in the diagram depicts the stability of the fixed points, and
the diagram indicates the existence of three bifurcation points: a
saddle-node and a Neimark-Sacker bifurcation on the nontrivial branch,
which itself bifurcates from the trivial branch at a transcritical
bifurcation. While subsequent sections of this paper will be used to
verify the bifurcation diagram using computer-assisted proofs, the
remainder of the current subsection is devoted to the discussion of
dynamical aspects which are observed through numerical simulations.
\begin{figure}[tb] \centering
  \setlength{\unitlength}{1 cm}
      \includegraphics[height=7cm]{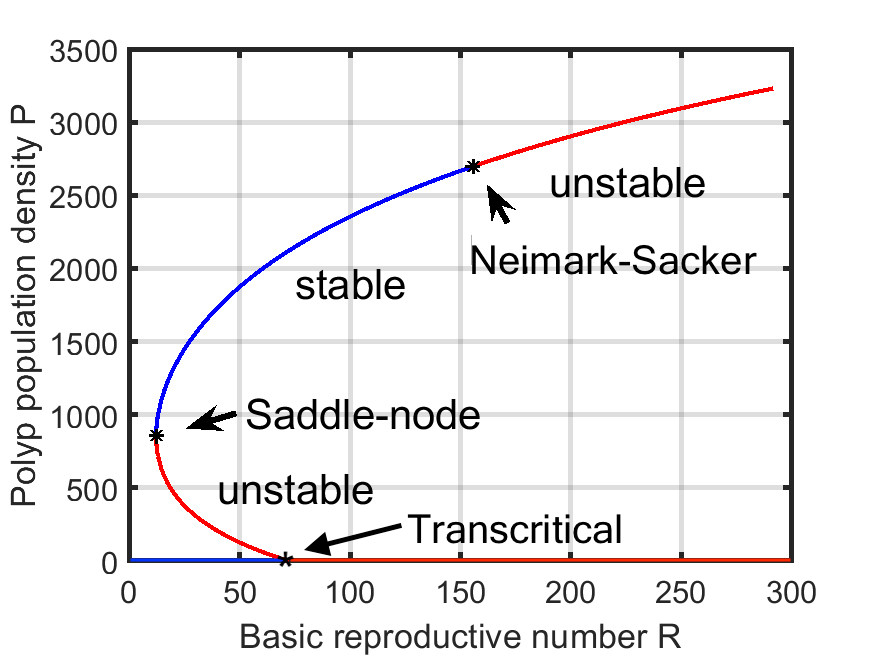}
  \caption{The bifurcation diagram of polyp density~$P$ as a function of
           the reproductive number~$R$. While the diagram covers the range
           $R \in (12,300)$, the birth rate data collected by Bramanti et
           al.\ in~\cite{bramanti:etal:09a} are for $R \approx 29$.}  
  \label{fig:bif}
\end{figure}%

Throughout our computations, we used the case of $d = 13$ age groups.
The bifurcation diagram in
Figure~\ref{fig:bif} was computed using a numerical continuation method
starting at reproduction number~$R=300$, and allowing~$R$ to decrease. There
appears to be a saddle-node point for $R \approx 12.28$ (which corresponds to
$\lambda \approx 0.4213$), after which the basic reproduction number~$R$ of
the fixed points begins to increase again. In Section~\ref{sec:validation} we
use a computer-assisted proof to rigorously validate this saddle-node bifurcation
point. The curve continues further until the population density reaches zero,
which corresponds to an extinct population. We will see later that the extinction
point can be found explicitly, and that it occurs at $R \approx 72.22$ (which
corresponds to $\lambda \approx 2.478$). Moreover, the stability of the trivial
solution $x = 0$ can readily be determined from the Jacobian matrix of~$f$ at
the origin, and this shows that the extinction fixed point is stable for small~$R$,
corresponding to low fitness, and unstable for all larger values of the basic
reproduction number~$R$, with instability index 1. The bifurcation between the 
extinction fixed point being stable and  unstable occurs at the transcritical bifurcation point.  
All of these statements will
be established rigorously in Section~\ref{sec:validation}, including the appearance
of the transcritical bifurcation point. Unlike the other two bifurcation points,
no computer-assisted proofs are necessary along the trivial solution.

As mentioned before, the stability of the fixed points 
$x^* \in \R^{13}$ is indicated by color, with blue indicating stable fixed
points and red representing unstable ones. The local stability at each fixed point in
Figure~\ref{fig:bif} is determined numerically, based on whether all the
eigenvalues of the Jacobian matrix~$D_x f (\lambda,x)$ lie inside the unit
circle or not. In the bifurcation diagram, we have not distinguished the 
index of the stability. If at least one of the eigenvalues lies outside the
complex unit circle, then the fixed point is colored red, meaning
unstable. 

\subsection{Oscillations}\label{subsec:rotation}

Figure~\ref{fig:bif} only shows the existence and stability behavior of fixed
point solutions. But what about the dynamical behavior of the system? In this 
last subsection of Section~\ref{sec:model}, we focus on dynamical aspects of
the model, in particular its oscillatory behavior on attracting invariant circles that 
form as a result of the Neimark-Sacker bifurcation. For a fixed parameter value $R>154.1$
and for a  typical initial condition, solutions converge to these invariant circles,
and therefore the age-structured coral populations oscillate as time varies. 

Figure~\ref{fig:sim} shows the dynamics of initial populations
near fixed points, starting at a variety of different parameters and different
initial aged-structured population vectors $y \in \R^{13}$. At reproduction
number~$R=8.744$ (which corresponds to $\lambda = 0.3$), the solutions converge
to the stable fixed point zero, i.e., the point of extinction. For $R=29.15$
(corresponding to $\lambda=1$), if we start at initial conditions ranging roughly
from~$0.15y$ to~$2y$, where~$y$ is a vector of age-structured initial number of
colonies which was chosen with polyp population density~$P = 1500$, then solutions
converge to a nontrivial stable fixed point. There is also an unstable fixed point
denoted by the red line. In addition, one can observe bistability at this parameter
value. If we start at a smaller value of~$P$, such as for example at initial populations
with polyp population density smaller than~$0.15y$, solutions converge to zero,
i.e., the coral population becomes extinct. At the basic reproduction number
$R=87.4437$ ($\lambda=3$), though it takes longer time than~$100$ years,
the solutions still converge to a stable nontrivial fixed point. In contrast,
at $R=160.31$ ($\lambda=5.5$), population starting at $P = 1.5y$ oscillate.
We used connected lines to show these oscillations more effectively, but recall
that the map is in fact discrete.
%
%
\begin{figure}[tb] \centering
  \setlength{\unitlength}{1 cm}
  \begin{picture}(15.5,10.1)
    \put(0.0,0.0){
      \includegraphics[width=7.25cm]{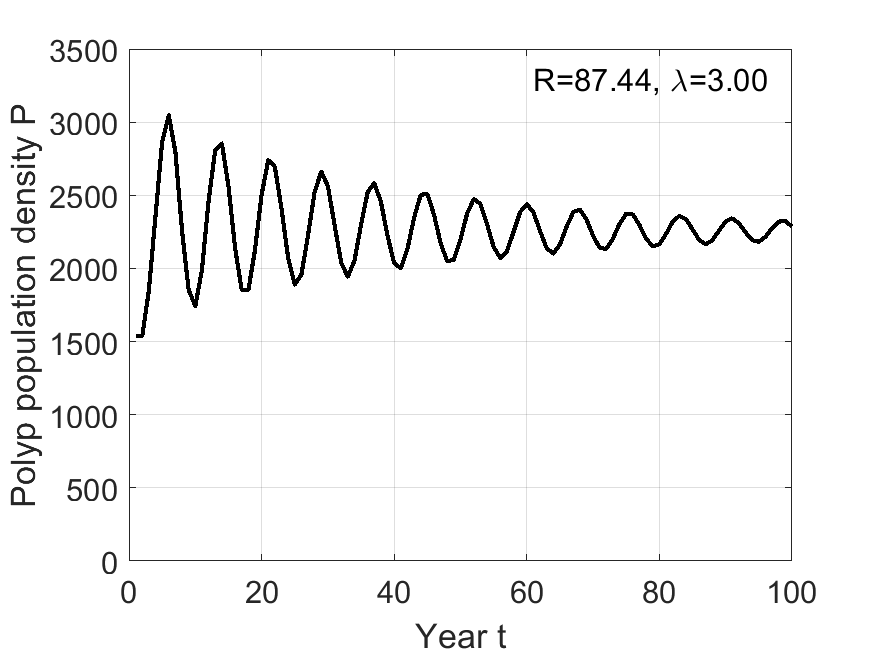}}
    \put(8.25,0.0){
      \includegraphics[width=7.25cm]{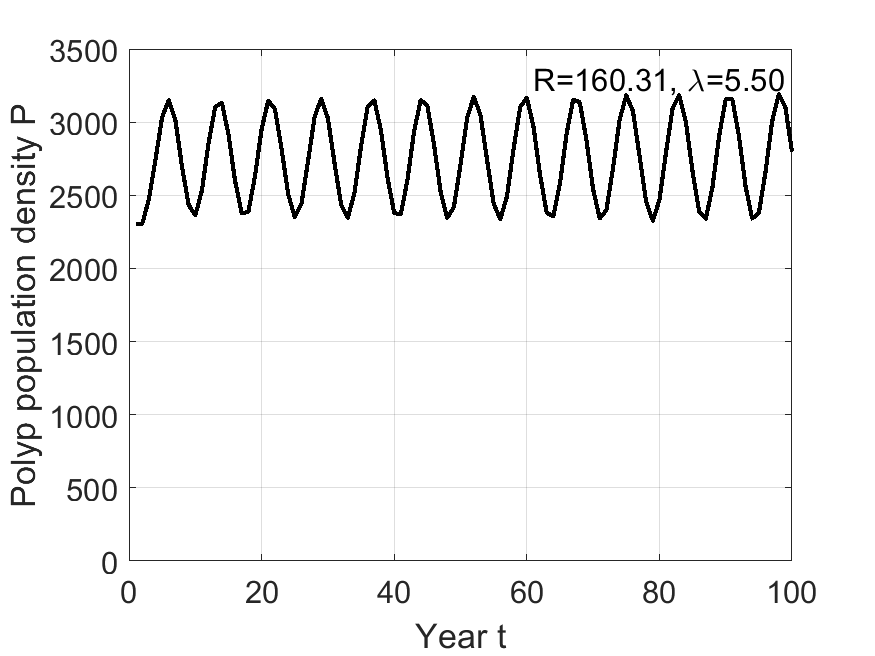}}
    \put(0.0,5.3){
      \includegraphics[width=7.25cm]{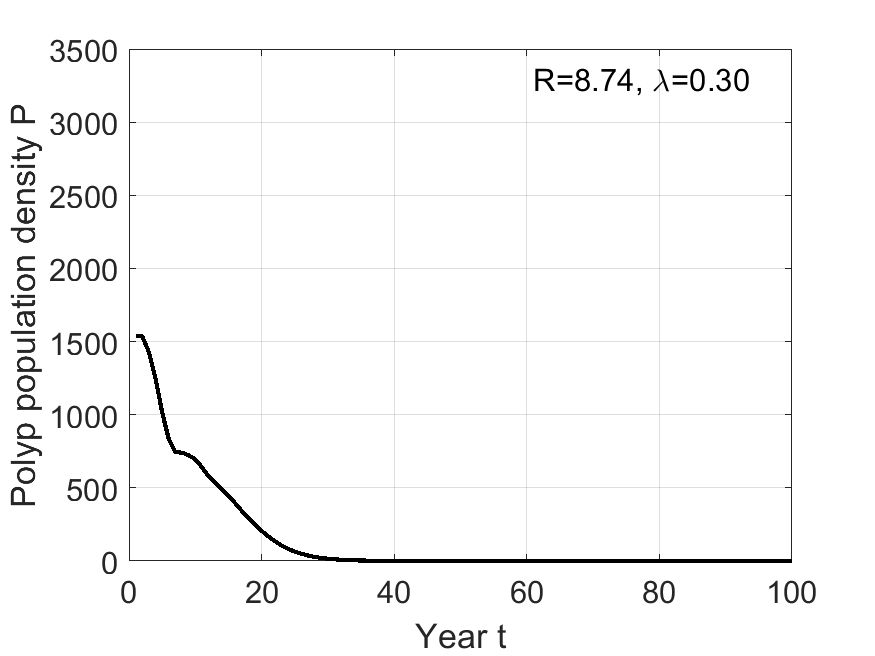}}
    \put(8.25,5.3){
      \includegraphics[width=7.25cm]{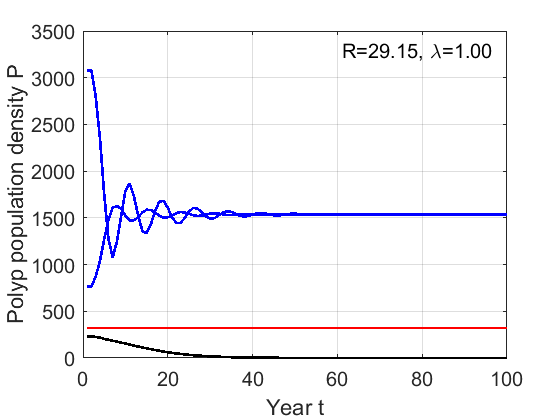}}	  
  \end{picture}
  \caption{Dynamical behavior of some sample orbits of the red coral 
           population model. All of these figures show the temporal evolution 
           of the polyp population density~$P$, and they are simulated over 
           a time frame of 100 years each, at various parameter values.}
  \label{fig:sim}
\end{figure}%
%
\begin{figure}[tb] \centering
  \setlength{\unitlength}{1 cm}
  \begin{picture}(15.5,5.5)
    \put(0.0,0.0){
      \includegraphics[width=7.25cm]{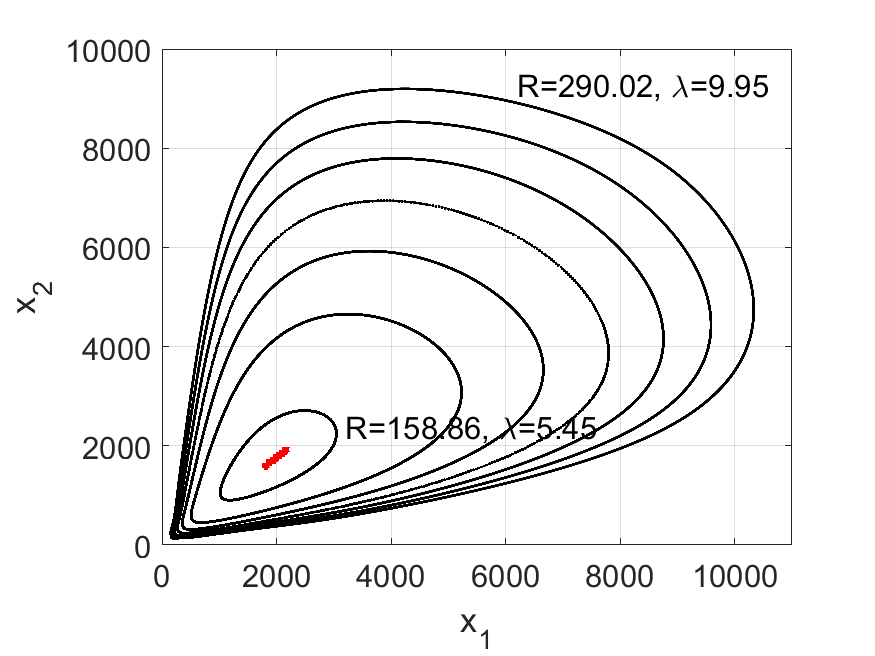}}
    \put(8.25,0.0){
      \includegraphics[width=7.25cm]{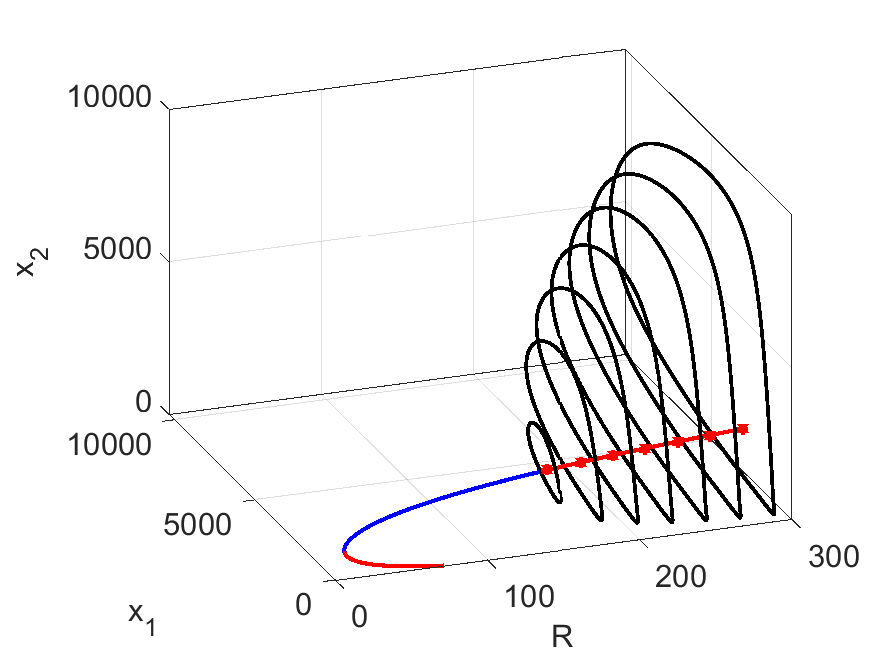}}
  \end{picture}
  \caption{After the Neimark-Sacker bifurcation, oscillating orbits appear. 
  After removing transients in the orbit, the orbit lies on an invariant closed curve. 
  On the left, we plot the~$x_1$- and $x_2$-components of these limit cycles. 
  As the parameter $R$ increases, 
  the size of the closed curve increases. For large values of~$R$, 
  the coral population is close to the extinction point at the origin.
  On the right, the same orbits are shown with respect to $R$, along with the corresponding unstable
  fixed points at the same parameter value.}
  \label{fig:hopf}
\end{figure}%
\begin{figure} \centering
\includegraphics[width=0.45\textwidth]{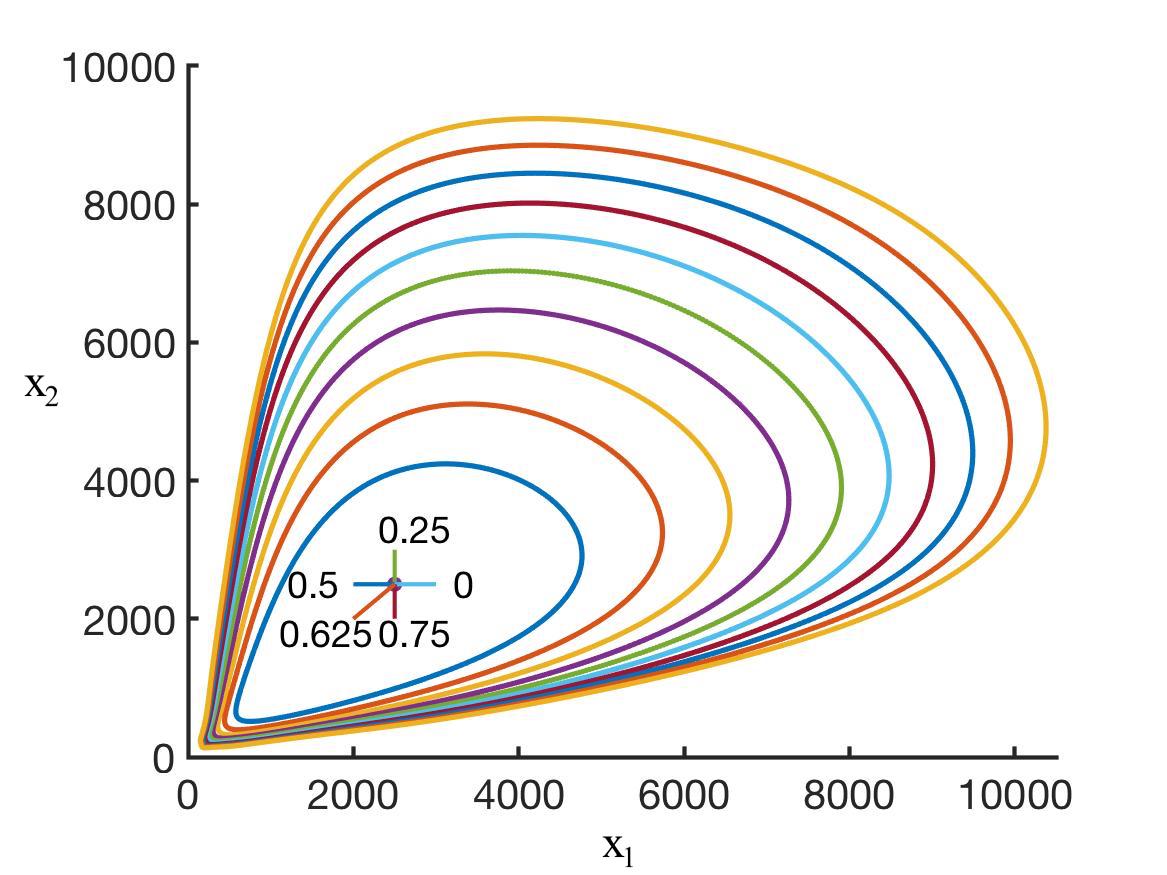}
\includegraphics[width=0.45\textwidth]{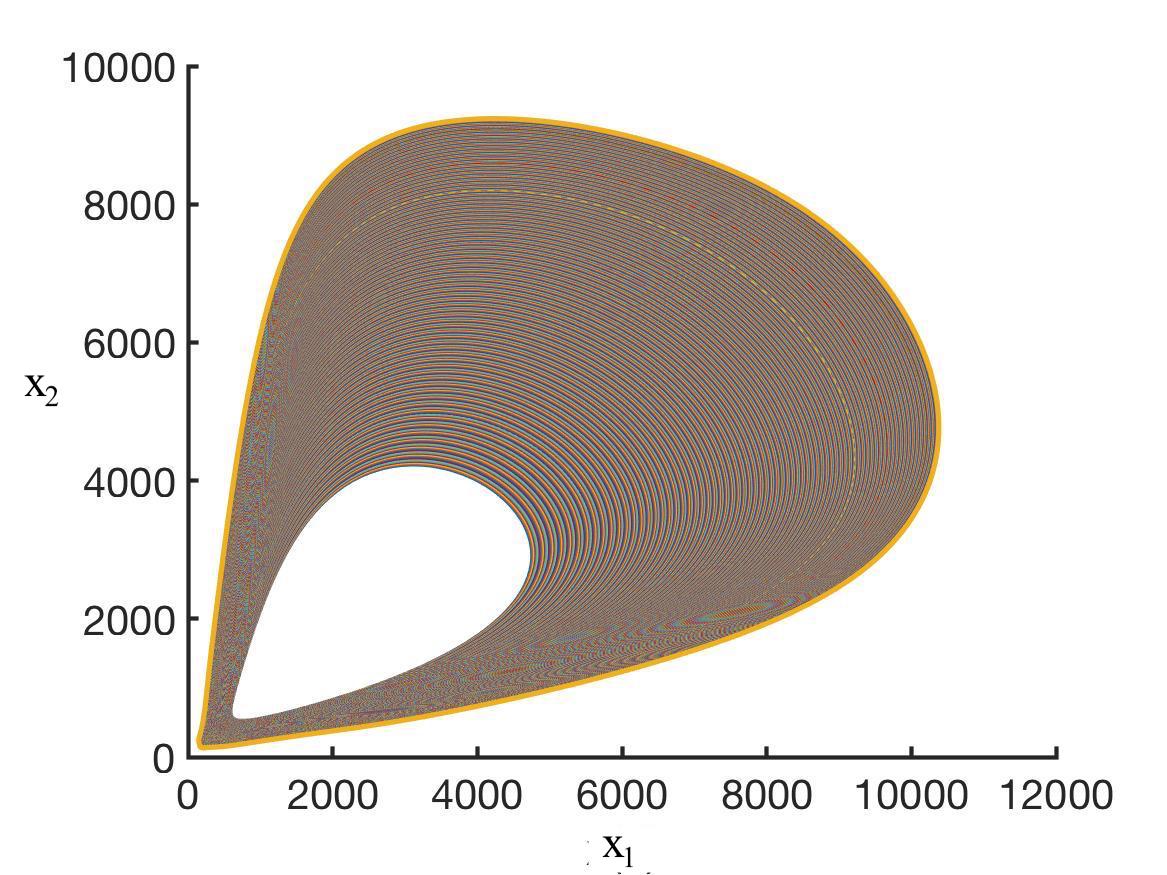}
\caption{\label{fig:cycles}
Top: Invariant cycles for  ten (left) and 500 (right) different parameter values.
 Even though we are guaranteed that some of the cycles contain stable periodic 
orbits, the periods are sufficiently high and the parameter ranges for which they 
exist are sufficiently small that it is hard to see them even in a close zoom (not depicted). 
Each orbit was computed using 100,000 iterates.}
\end{figure}

The oscillations seen in the lower right subplot of Figure~\ref{fig:sim}  
form as a result of  the Neimark-Sacker bifurcation. The fixed point stability
switches from stable to unstable, and an invariant circle gains stability. 
Trajectories with initial conditions near fixed points but after the 
bifurcation are displayed in Figure \ref{fig:hopf}. Perturbations around
an unstable fixed point are repelled from the fixed point 
after the bifurcation, converging to an invariant  closed curve. As the
parameters~$R$ and~$\lambda$ increase, the size of the closed curve also
increases, and the minimum population of a curve approaches the
extinction point at the origin. That is, red coral populations become
vulnerable at a large reproduction number, and a very small perturbation
of the population would endanger the survival of the population despite the
existing long recovery cycle.
\begin{figure}\centering
\includegraphics[width=0.32\textwidth]{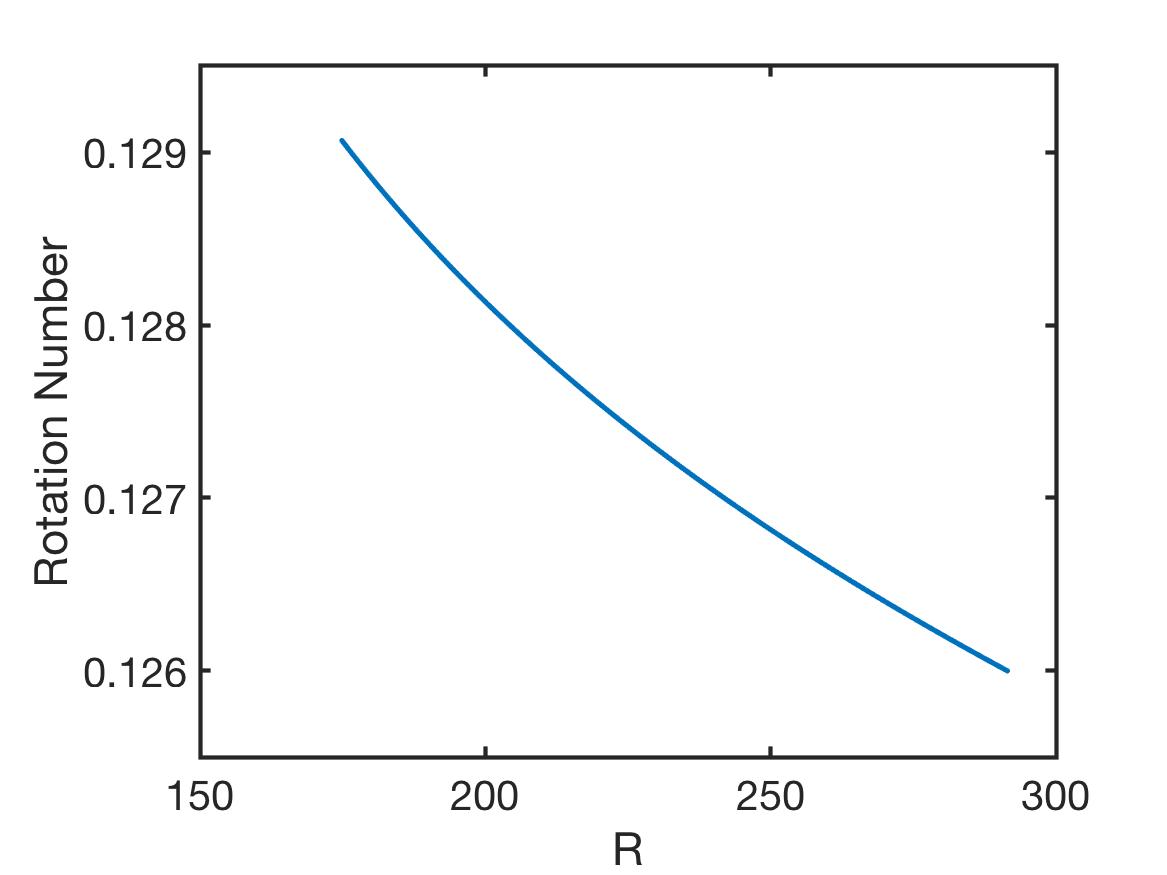}
\includegraphics[width=0.32\textwidth]{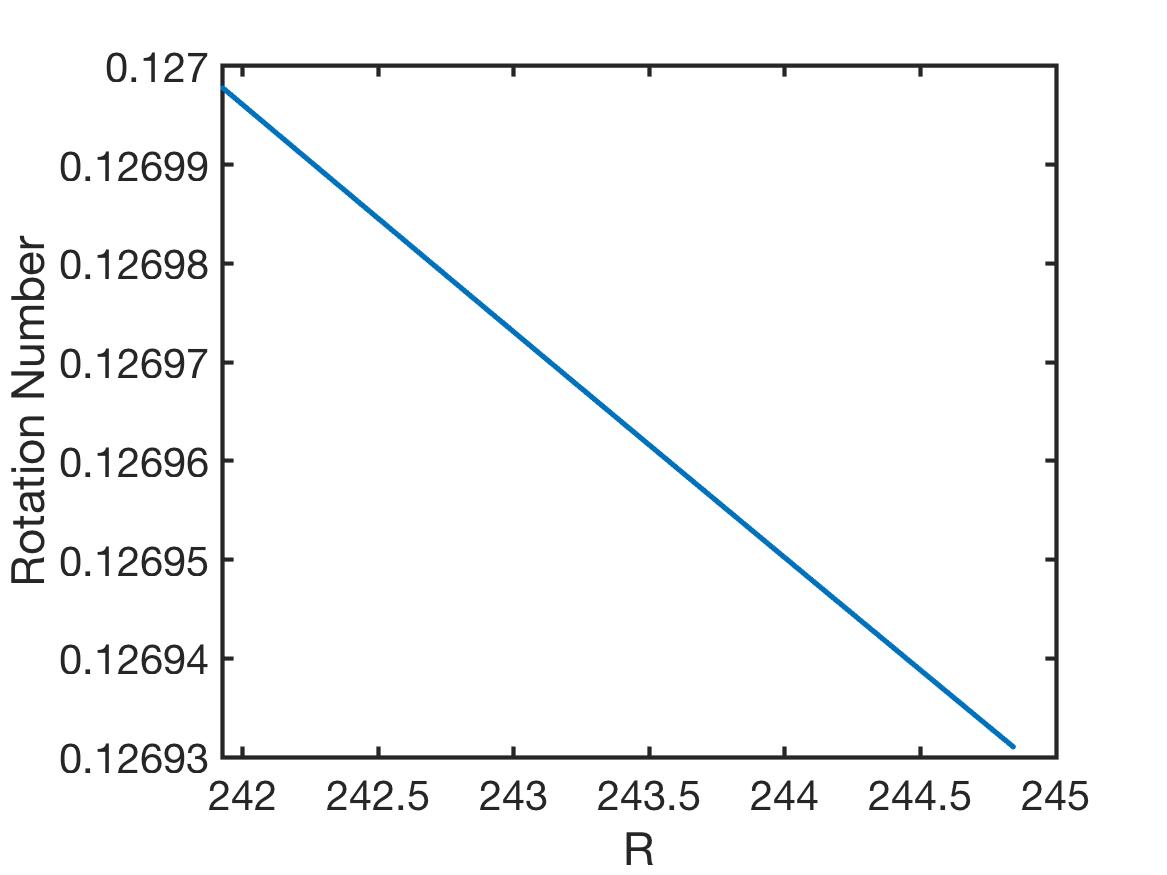}
\includegraphics[width=0.32\textwidth]{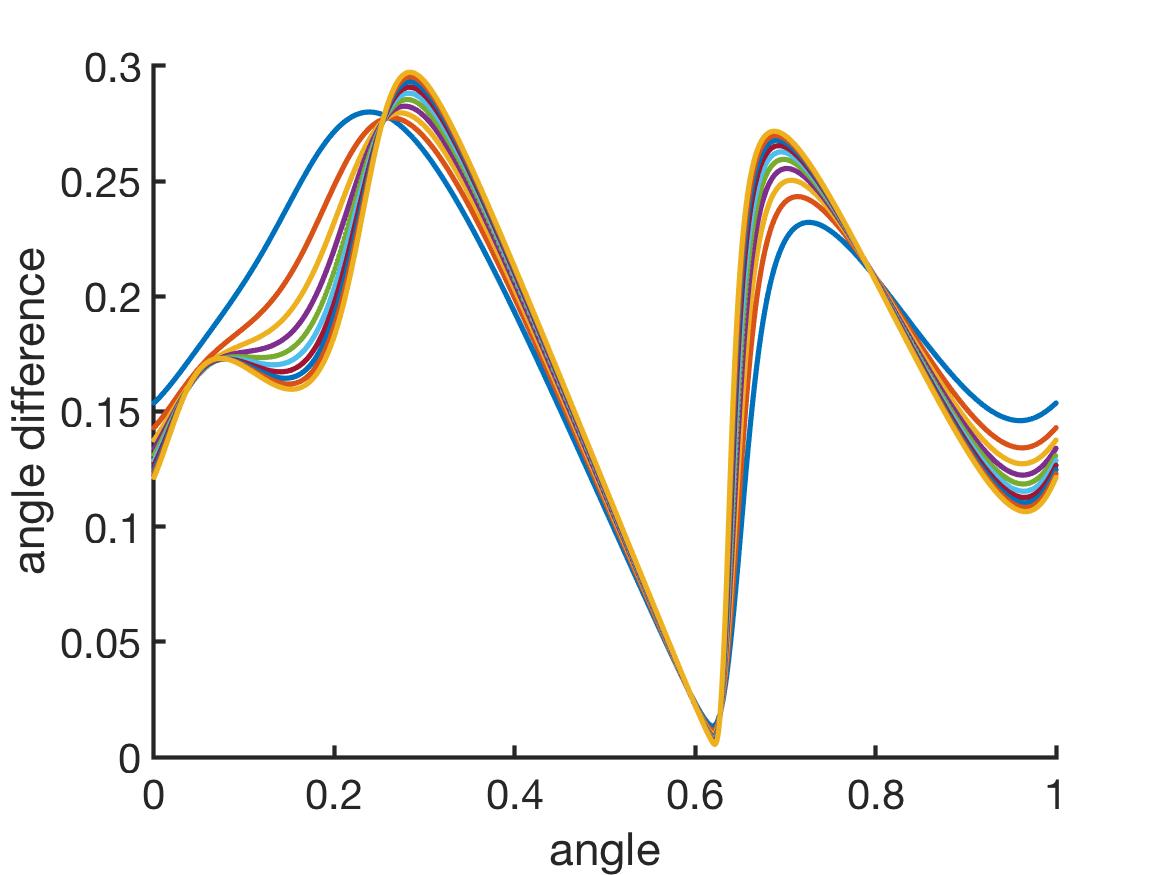}
\caption{ \label{fig:rotnum}
  The rotation number for the cycles shown in Figure~\ref{fig:cycles} (left)
  and a close up view of the rotation numbers (middle), this time with one
  million iterates. The periodic orbits are of such high periods that 
  we cannot detect the devil's staircase type behavior of the rotation number
  within the Arnold tongue locking regions. The rotation number is computed
  using the angle difference between successive values of~$(x_1,x_2)$, computed
  with respect to the point $(x_1,x_2)=(2500,2500)$, and the angle versus angle
  difference is depicted here (right) for the cycles for ten different~$R$ values.
  The minimum occurs at the angle pointing towards the extinction point.} 
\end{figure}

In order to better understand the stable invariant limit cycles that form after
bifurcation, we have computed the rotation number, meaning the average angle of
rotation per iterate, as a function of the parameter~$R$. Specifically, we used
the projection to the $x_1  x_2$-plane to compute the rotation numbers. Our
computations are performed using the weighted Birkhoff average method described
in~\cite{das:etal:16a}. Figure~\ref{fig:cycles} shows cycles at a ten distinct parameter
values on the left, and for 500 distinct parameters on the right. The corresponding
rotation numbers are shown in Figure~\ref{fig:rotnum}. The values are angles, but
they are rescaled to have values in the range $(0,1)$. Each rotation number was
computed by considering the angle difference between successive iterates when
measured with respect to the point $(2500,2500)$. To verify our numerics and
check that we have used a sufficient number of iterates in our calculation,
we compared the rotation number computed with 50,000 iterates to the rotation
number computed with 40,000 for a series of test parameters. In these test
parameters, the answer differs by $10^{-15}$ or less. 

Note that we would expect to see a devil's staircase in the rotation 
numbers at the parameter values when there are periodic orbits, 
but what we see looks smooth even when quite zoomed in. This is due to the fact that 
the periodic orbits are extremely high period. In particular, we are able to use a Farey 
tree calculation to find the smallest denominator, corresponding to the lowest period, of a periodic orbit 
for the case of a rational rotation number for this range of rotation numbers, using the 
method in~\cite{beslin:etal:98a,sander:meiss:20a}. In particular, we find 
that the lowest denominator in the range $[0.126,0.129]$ is 39 (fraction $5/39$). 
See the zoomed in look at the cycles in the bottom
two images in Figure~\ref{fig:cycles}. Not only is the lowest possible period quite large and 
therefore hard to distinguish from a limit cycle, but also the large periodicity
implies that the Arnold tongue locking regions are 
very small parameter ranges, meaning that we are not able to resolve them without more delicate computations. 

The average rotation number gives only the mean of how much the population is changing 
with respect to time. This leaves out some information as to how the change in population 
depends on the location of the population. In the right subplot in Figure~\ref{fig:rotnum}, 
we show the angle difference as a function of the angle for ten different values of~$R$. 
That is, for each point in the invariant circle, 
we graph how much the population is changing in one iterate (corresponding to one 
year) at each point in the invariant circle.  The 
smallest angle difference, corresponding to the slowest change, occurs for angle $\approx 0.625$,
corresponding to the values closest to the origin extinction point. Therefore, 
a portion of the invariant circles is getting dangerously close to the origin, such that a 
small perturbation could result in the extinction of the whole coral population. To compound
matters further, the orbits are staying near 
the extinction point for longer than they remain in any other region, since
at these points the observed angle differences are very close to zero. Thus the population 
remains extremely vulnerable for a particularly long time.

\section{Branch validation and continuation}\label{sec:arclength}

We now turn to the rigorous validation of fixed points, both for regular and
bifurcation values. Our general approach is the constructive implicit function
theorem from~\cite{sander:wanner:16a}. This is a rigorous result that combines
with a numerical interval arithmetic calculation to give rise to a validated
method for finding a branch in the zero set of a function which depends on a
single parameter. In the following four subsections, we will first recall the
constructive implicit function theorem, and then define an extended system
which can be used for pseudo-arclength continuation. After that, we prove
two results which form the basis of our approach, and describe the necessary
preconditioning for the coral population model application.

\subsection{The constructive implicit function theorem}

Before stating the full result, here is a summary. 
Given an approximate zero $(\alpha^*,x^*)$ of a function $G(\alpha,x)$ 
where $x$ is contained in a Banach space and $\alpha \in \R$,  under certain
hypotheses on $G$ and its derivatives evaluated at the approximate zero $(\alpha^*,x^*)$, 
combined with Lipschitz estimates near this point,
there exist two regions in parameter and phase space. First, the {\em accuracy
region\/}, which contains a curve of the zero set.  Second, a {\em uniqueness
region\/}, in which that zero set curve is unique. 
See the schematic in Figure~\ref{fig:uniqueacc}. The blue dot 
shows the initial approximate zero. The orange curve is the 
zero set curve, which is guaranteed to lie within the accuracy region (the blue region). 
Note that the approximate zero does not in general lie on the zero set. The 
accuracy region is contained within the 
uniqueness region, shown in orange. 
The uniqueness region is largest in phase space when the 
parameter is closest $\alpha^*$. As the parameter varies, the 
uniqueness region shrinks (meaning we have worse isolation). 
The constructive implicit function theorem guarantees that the 
uniqueness region is characterized by a linear norm condition, as 
depicted by the straight sides in the schematic diagram. The accuracy region 
has best (i.e., smallest)  
accuracy when the parameter is near the parameter of the original point~$\alpha^*$.  
The accuracy region grows (meaning we have worse accuracy) with a quadratic norm condition. 
This is depicted schematically by its parabolic shape. 
We now state the formal theorem. 
\begin{figure}[tb] \centering
      \includegraphics[width=0.4\textwidth]{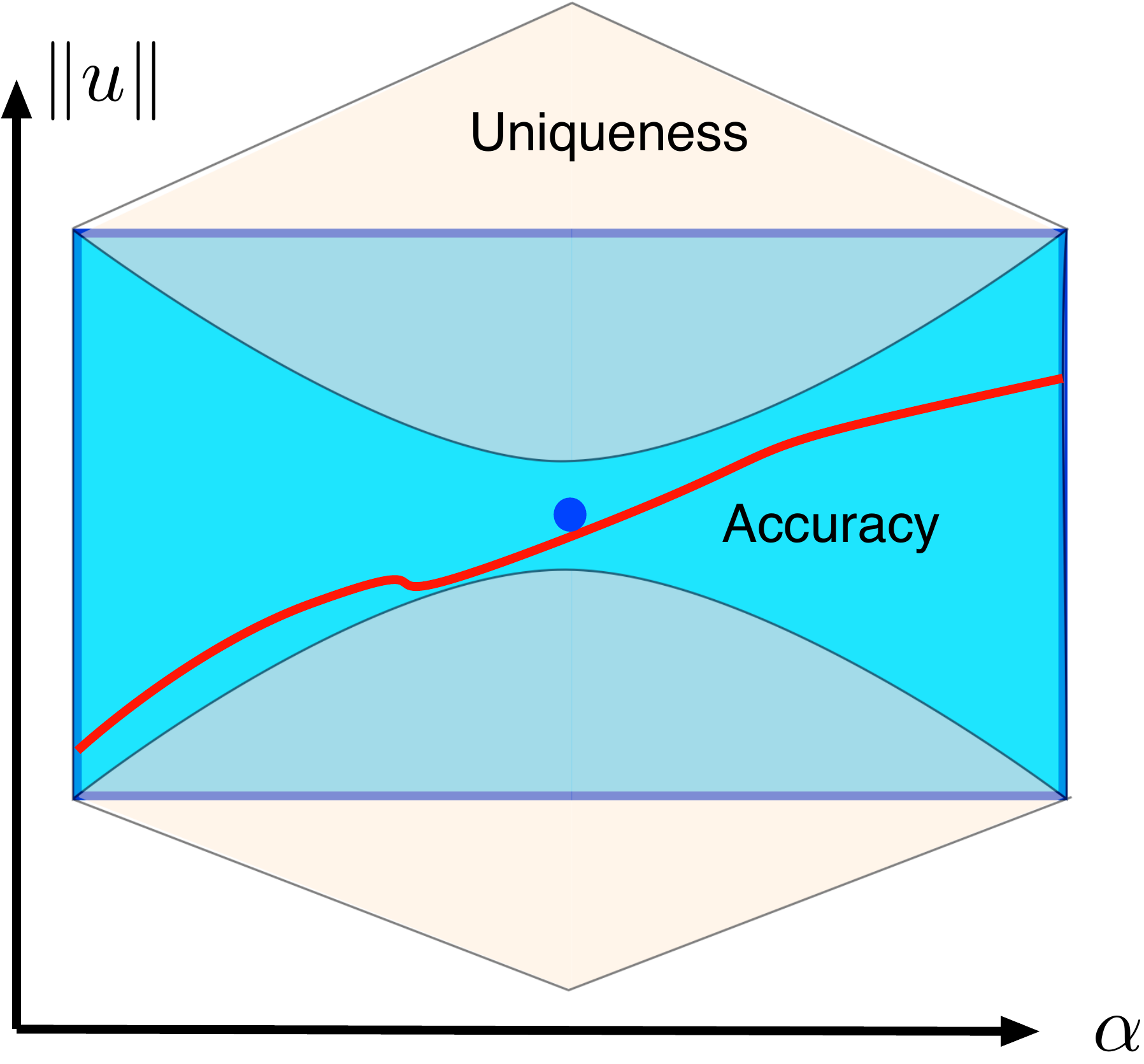}
  \caption{A schematic depiction of the constructive implicit function theorem. 
  The theorem  guarantees that under 
  appropriate hypothesis, an approximate zero (blue dot) guarantees that within a uniqueness 
  region (orange region) there is a curve in the zero set with a unique point at 
  each fixed $\alpha$ value (red curve), and the 
  this curve is located within an accuracy region (blue region). The uniqueness
  region contains the accuracy region. It is bounded in norm by straight lines, 
  and the accuracy region is bounded in norm by parabolas.}
  \label{fig:uniqueacc}
\end{figure}%

\begin{theorem}[Constructive Implicit Function Theorem]
\label{nift:thm}
Let~$\cP$, $\cX$, and~$\cY$ be Banach spaces, suppose that the
nonlinear operator $G : \cP \times \cX \to \cY$
is Fr\'echet differentiable, and assume the following
hypotheses.
\begin{itemize}
\item[(H1)] Small residual: There exists a pair
$(\alpha^*,x^*) \in \cP \times \cX$ and a $\rho>0$ such that
\begin{displaymath}
  \left\| G(\alpha^*,x^*) \right\|_\cY \le \rho \; .
\end{displaymath}
\item[(H2)] Bounded derivative inverse: There exists a constant~$K > 0$
such that
\begin{displaymath}
  \left\| D_x G(\alpha^*,x^*)^{-1} \right\|_{\cL(\cY,\cX)} \le K
  \; ,
\end{displaymath}
where~$\| \cdot \|_{\cL(\cY,\cX)}$ denotes the operator norm
in~$\cL(\cY,\cX)$. 
\item[(H3)] Lipschitz bound:  There exist
positive real constants~$L_1$, $L_2$, $\ell_x$,
and~$\ell_\alpha \ge 0$ such that for all pairs $(\alpha,x)
\in \cP \times \cX$ with $\| x - x^* \|_\cX \le \ell_x$ and
$\|\alpha - \alpha^*\|_\cP \le \ell_\alpha$ we have
\begin{displaymath}
  \left\| D_x G(\alpha,x) -
    D_x G(\alpha^*,x^*) \right\|_{\cL(\cX,\cY)} \le
    L_1 \left\| x - x^* \right\|_\cX +
    L_2 \left\|\alpha - \alpha^* \right\|_\cP \; .
\end{displaymath}
\item[(H4)] Lipschitz-type bound: There exist positive
real constants~$L_3$ and~$L_4$, such that for all parameters $\alpha \in \cP$
with $\|\alpha - \alpha^*\|_\cP \le \ell_\alpha$ one has
\begin{displaymath}
  \left\| D_\alpha G(\alpha,x^*) \right\|_{\cL(\cP,\cY)} \le
    L_3 + L_4 \left\| \alpha - \alpha^* \right\|_\cP \; ,
\end{displaymath}
where~$\ell_\alpha$ is the constant that was chosen in~(H3).
\end{itemize}
Finally, suppose that
\begin{equation} \label{nift:thm1}
  4 K^2 \rho L_1 < 1
  \qquad\mbox{ and }\qquad
  2 K \rho < \ell_x \; .
\end{equation}
Then there exist pairs of constants~$(\delta_\alpha,\delta_x)$ with
$0 \le \delta_\alpha \le \ell_\alpha$ and $0 < \delta_x \le \ell_x$,
as well as
\begin{equation} \label{nift:thm2}
  2 K L_1 \delta_x + 2 K L_2 \delta_\alpha \le 1
  \qquad\mbox{ and }\qquad
  2 K \rho + 2 K L_3 \delta_\alpha + 2 K L_4 \delta_\alpha^2
    \le \delta_x  \; ,
\end{equation}
and for each such pair the following holds. For every~$\alpha \in \cP$
with $\|\alpha - \alpha^*\|_\cP \le \delta_\alpha$ there exists a uniquely
determined element~$x(\alpha) \in \cX$ with $\| x(\alpha) - x^* \|_\cX
\le \delta_x$ such that $\cG(\alpha, x(\alpha)) = 0$.
In other words, if we define
\begin{displaymath}
  \cB_\delta^\cX = \left\{ \xi \in \cX \; : \;
    \left\| \xi - x^* \right\|_\cX \le \delta \right\}
  \quad\mbox{ and }\quad
  \cB_\delta^\cP = \left\{ p \in \cP \; : \;
    \left\| p - \alpha^* \right\|_\cP \le \delta \right\}
  \; ,
\end{displaymath}
then all points of the solution set of the equation $G(\alpha,x)=0$ in the
set $\cB_{\delta_\alpha}^\cP \times \cB_{\delta_x}^\cX$ lie on the graph
of the function $\alpha \mapsto x(\alpha)$.  
\end{theorem}
\begin{figure}[tb] \centering
  \setlength{\unitlength}{1 cm}
      \includegraphics[width=7.cm]{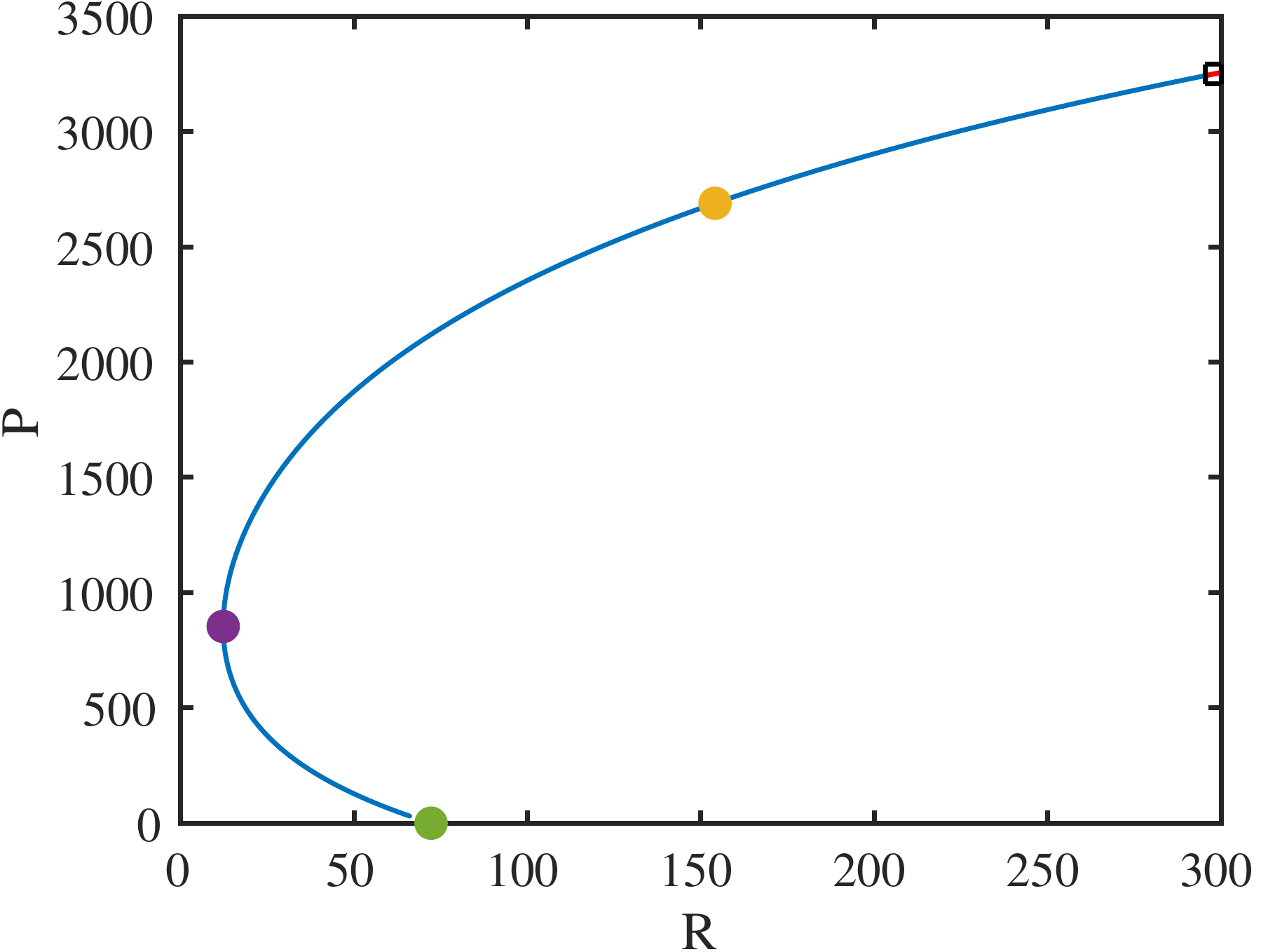}
       \includegraphics[width=7cm]{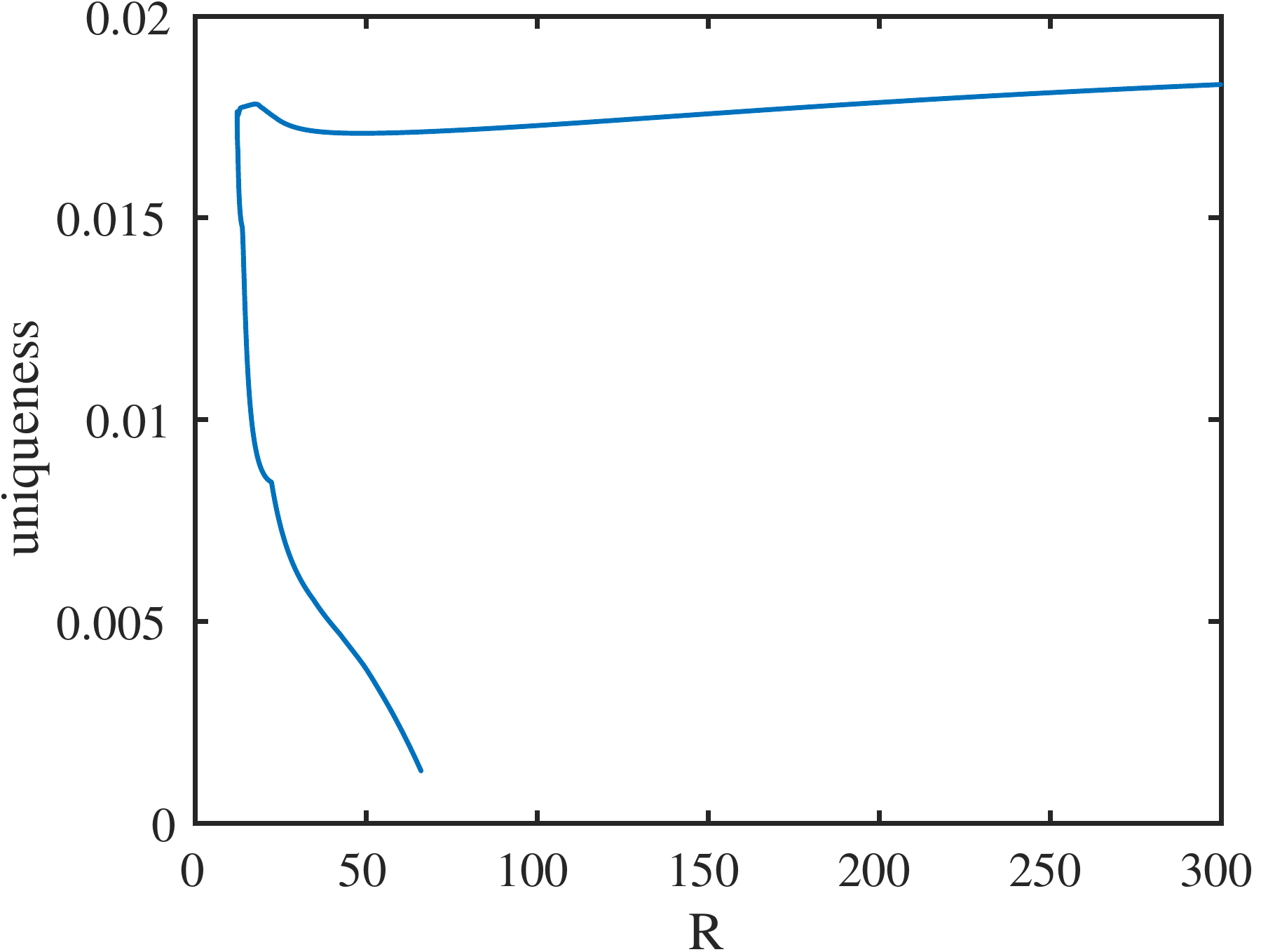}
  \caption{Left, the validated bifurcation diagram of polyp density~$P$ as a
  function of the  reproductive number~$R$, 
  along with the three validated bifurcation points. 
  The blue curve consists of 5000 continuation steps, corresponding to 5000 linked boxes, 
  for the preconditioned map with $\alpha = 0.8 \; \delta_\alpha$. 
  The initial validated box contains $(R,P) = (300, 3256)$, which is in the upper right
  corner of the bifurcation diagram, and the last validated box contains $(R,P) = (71.91, 1.493)$,
  which is close to the green transcritical bifurcation point.
  For comparison purposes, 4000 continuation steps for the unconditioned map are 
  shown in red within the extremely 
  small square region in the upper right-hand corner. 
  Right, the norm of the uniqueness region of the solution. As the 
  solution gets near the transcritical bifurcation at the origin, the uniqueness region gets smaller. 
  This is expected,  since there is no longer any uniqueness 
  when the two branches of the solution curve meet.
  }  
  \label{fig:comp}
\end{figure}%
In its classical form, the implicit function theorem is one of the central
tools of bifurcation theory. Not only can it be used to establish the existence
of small solution branches in nonlinear parameter-dependent equations, but by
applying it as a tool to modified problems it can frequently be used to provide 
sufficient conditions for bifurcations. For example, the celebrated Crandall-Rabinowitz
result~\cite{crandall:rabinowitz:71a} on bifurcation from a simple eigenvalue proves
the existence of a bifurcating branch by applying the implicit function theorem
to a modification of the original nonlinear problem which removes the trivial
solution. The constructive implicit function theorem can
similarly be used as a tool for bifurcation analysis, yet in a computer-assisted proof
setting. In fact, some first applications in this direction have already been provided
in~\cite{lessard:sander:wanner:17a, sander:wanner:16a}. With the current paper, we add
two more applications.

More precisely, in the following we will be applying Theorem~\ref{nift:thm} in two
different situations. In the remainder of this section, we apply it for branches of
regular points. Through the introduction of a suitable extended system we can reformulate
a validated step of pseudo-arclength continuation as an application of the constructive
implicit function theorem to this extended system. Combined with suitable linking
conditions, this establishes the existence of entire branches covered by slanted boxes.

In addition, in Section~\ref{sec:validation} we use Theorem~\ref{nift:thm} to validate 
bifurcation points. In that setting, and motivated by our earlier work~\cite{lessard:sander:wanner:17a},
we will apply the theorem to an extended system without any parameter, as the parameter
will be incorporated into the function for which we find a root. This parameter-free
case means that we no longer need to find the Lipschitz constants relevant to the
parameter variations, and we set these unused constants equal to zero. 

\subsection{Continuation and an extended system}

To elaborate further on the validation of regular fixed points, the constructive implicit function
theorem as stated in~\cite{sander:wanner:16a}
only applies to a single region, validated at a single point.
The same paper contains a version of this theorem for slanted boxes, 
using natural continuation in order to validate a branch of solutions by linking their validation sets to validate a
larger portion of the branch. However, 
natural continuation leaves something to be desired in terms of efficiency. 
In this section, we develop a method of validation of bifurcation branches using 
pseudo-arclength continuation which allows for the direct application 
of the constructive implicit function theorem, and apart from Lipschitz estimates, 
only requires estimates at a single point  in each box. 
This method is an improvement on the previous natural continuation 
method in that we can continue at limit points without having to change coordinates. 
The methods in this section apply for regular orbits along branches. In the next section, we will 
show how to adapt the constructive implicit function theorem in order to rigorously validate bifurcation points. 

Before launching into further technicalities, we describe our results. Applying the pseudo-arclength continuation method to a preconditioned version of the
coral model (preconditioning is discussed in Section~\ref{sec:preconditioning} below), the resulting rigorously validated curve of fixed points is shown in Figure~\ref{fig:comp}. 
While Figure~\ref{fig:bif} shows a similar picture, the distinction is that those points were 
found using numerical methods, and though we have a priori error estimates for these methods, 
we cannot guarantee existence or accuracy. In contrast,  the points shown on the new figure are rigorously validated. 
The depicted points are an accurate indication of existing fixed points of the system, with known
and validated accuracy and uniqueness region. In particular, the accuracy of our solutions 
is known individually for each separate box, and is always less than $1.453 \cdot 10^{-13}$,  where the error in 
$x \in \R^{13}$ is measured in the maximum norm. Figure~\ref{fig:comp} shows the norm of the uniqueness  
for each separate box. The uniqueness shrinks when the 
curve approaches zero. This is not surprising, since $x=0$ is part of the zero set, 
putting a barrier on the size of the uniqueness region.

We now proceed with the constructive
implicit function theorem for a validated pseudo-arclength continuation.
In each continuation step we  use continuation in a box with slanted
sides, where the predictor step is performed along the middle of the box
in the direction a specified vector $(\mu,v)$ (usually the estimated tangent
to the zero set curve), and the corrector step uses a computation such as 
Newton's method to refine the estimate. This refinement is 
performed in a direction orthogonal to the
predictor direction $(\mu,v)$. This is depicted in Figure~\ref{fig:ac1}. The
left-hand image is a schematic diagram showing the box with its midline
between two blue dots. The midline is the estimated tangent line in the direction 
$(\mu,v)$. Our validation gives us a maximum length of the box for which we 
can guarantee accuracy and uniqueness of the solution. 
The predictor,   shown with a red dot, must be chosen inside that box. 
The corrector,  shown with a
green dot is along an orthogonal line to the midline. 
The right-hand image shows the accuracy region in blue and the
uniqueness region in orange. Note that the uniqueness region has large
width near the starting point, and the accuracy region grows towards the
ending point. In Figure~\ref{fig:ac1}, the uniqueness region for the box is approximately diamond 
shaped, whereas in Figure~\ref{fig:ac2}, the box is not only slanted but also has a uniqueness region 
which is asymmetric, more of a half-diamond. The half-diamond shape is in fact only half of 
the uniqueness box. In particular, as we are merely continuing in
one direction, which in Figure~\ref{fig:ac2} is to the left, we only show one side of the 
uniqueness box. The fact that we could continue to the right as well is not relevant for 
our continuation. 
\begin{figure}[tb] \centering
      \includegraphics[width=0.4\textwidth]{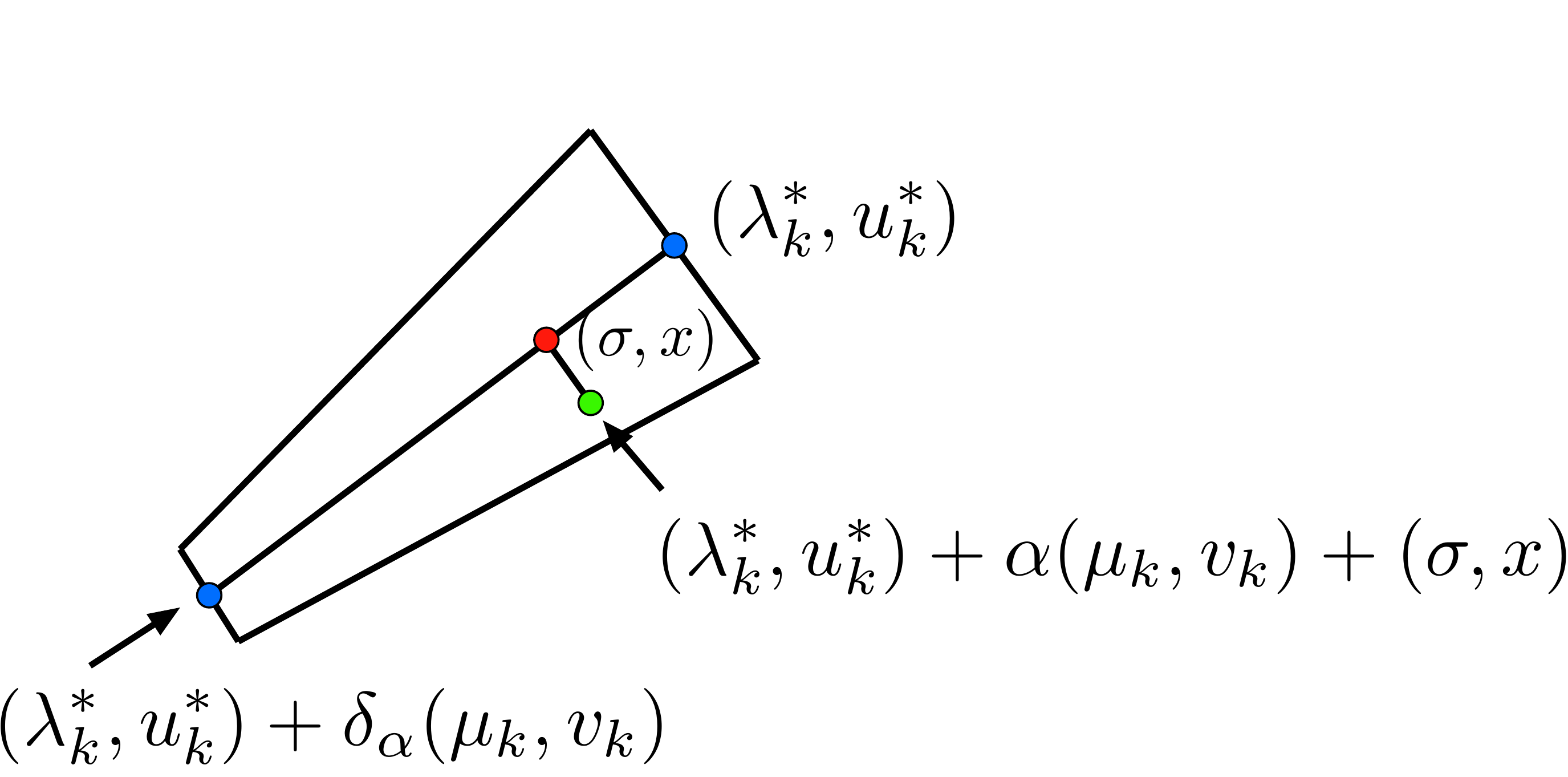}
      \includegraphics[width=0.3\textwidth]{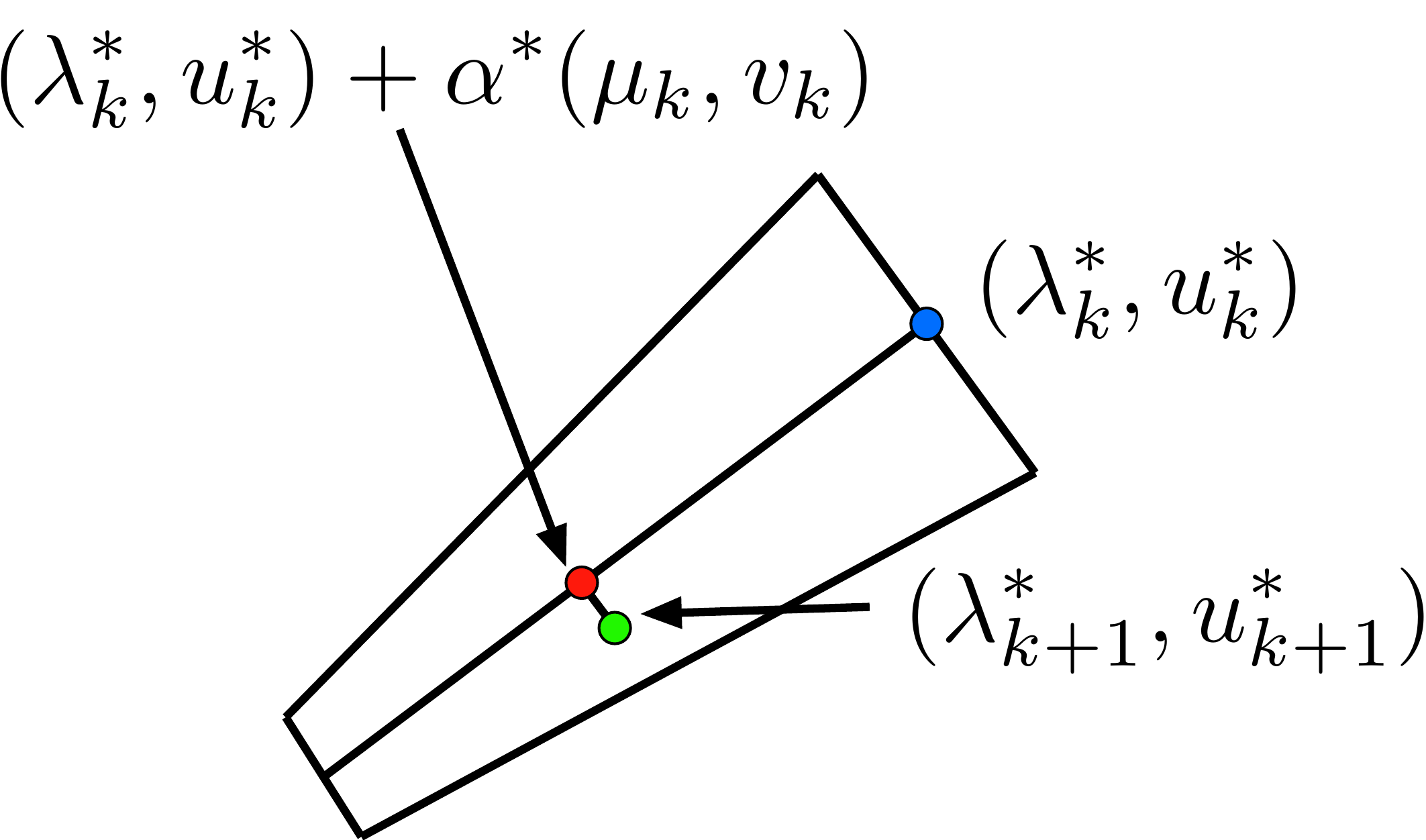}
      \includegraphics[width=0.25\textwidth]{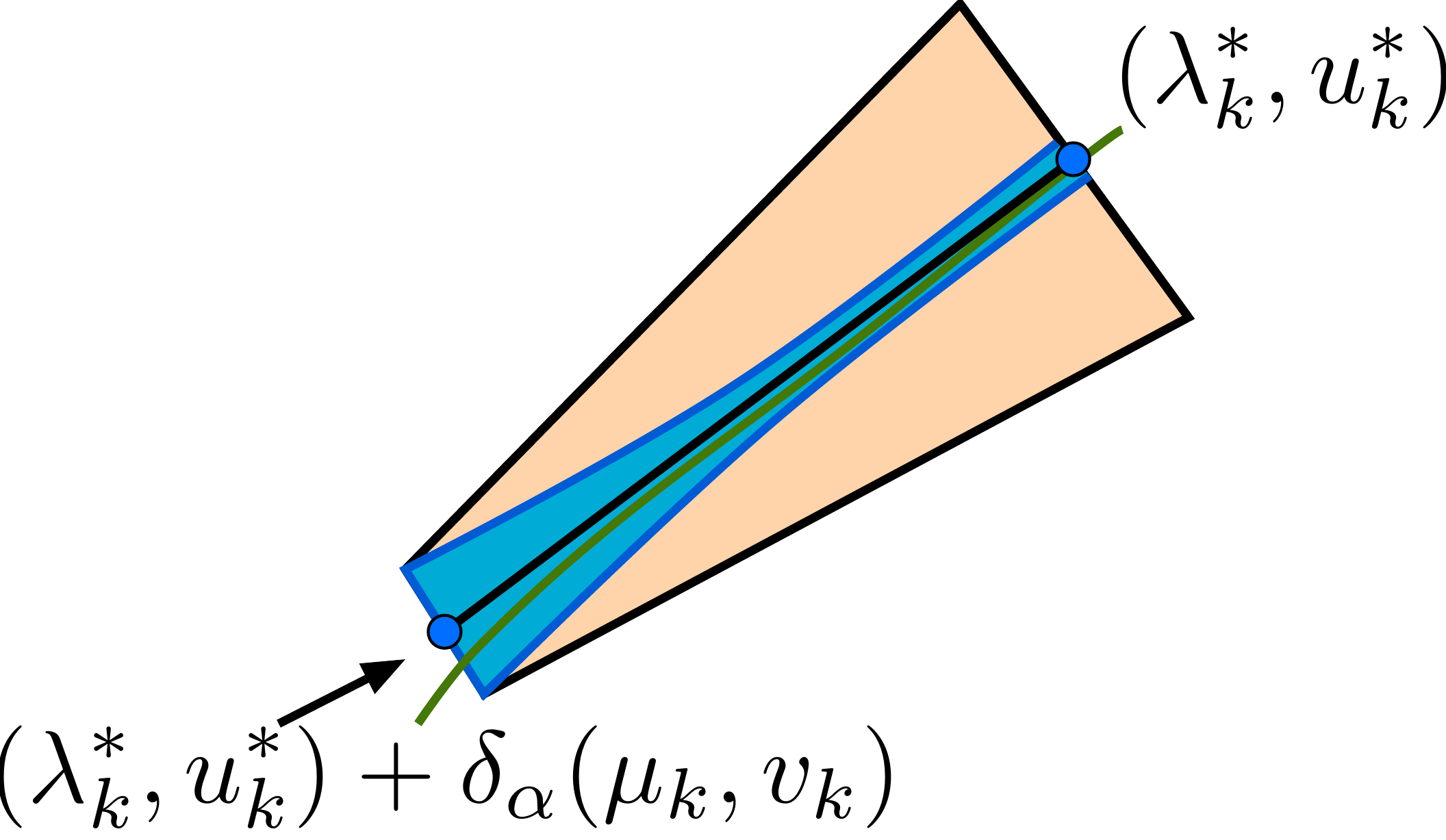}
  \caption{A schematic diagram of the the pseudo-arclength continuation method.
  Left image: The result guarantees a uniqueness region for the zero set. 
  This takes place in an adapted coordinate system, 
  meaning that the box is slanted, but the uniqueness region is 
  still bounded by straight lines.
  Since we only continue the curve in one direction, 
  this figure only depicts the left half of the uniqueness region. 
   The center line segment of this region is given by $(\lambda_k^*,u_k^*) + \alpha (\mu_k,v_k)$
   for $0 \le \alpha \le \delta_\alpha$. 
   At a fixed $\alpha$ value, we use Newton's method to find the next approximate zero
   along the line $(\lambda_k^*,u_k^*) + \alpha (\mu_k,v_k) + (\sigma,x)$, where~$(\sigma,x)$ 
   denotes the vector pointing from $(\lambda_k^*,u_k^*) + \alpha (\mu_k,v_k)$
   to the point~$(\lambda_k^*,u_k^*) + \alpha (\mu_k,v_k) + (\sigma,x)$, and which
   is orthogonal to~$(\mu_k,v_k)$.
   Middle image:  After we fixed the value $\alpha=\alpha^*$, 
   we label this next approximation $(\lambda_{k+1}^*,u_{k+1}^*)$.
   Right image: Inside the uniqueness region (orange) is an  accuracy region
  (blue). The  accuracy region is bounded by curves which are parabolic in norm in the 
  adapted coordinate system.
   }
  \label{fig:ac1}
\end{figure}%

We now turn to the technical details of this approach. For this, let
$F: \R \times U \to U$, where~$U$ denotes an arbitrary Euclidean space.
Our goal is to implement pseudo-arclength continuation based on Theorem~\ref{nift:thm}
to find branches of zeros of the nonlinear function~$F$. For the specific application
of this paper, we will consider $U = \R^{13}$ and $F(\lambda,x) = f(\lambda,x) - x$, 
where $f$ is the coral model. Nevertheless, we use the more general notation based on~$F$
to indicate that these methods are general. In fact, the methods readily generalize to
the Banach space setting as well. However, in this paper for
convenience of notation we only consider the Euclidean space case. For any
$(\lambda_0,u_0) \in \R \times U$, an approximate zero of~$F$, and 
for a fixed direction vector $(\mu_0,v_0) \in \R \times U$, define 
$G: \R \times (\R \times U) \to \R \times U$ as follows
\begin{equation} \label{defextG}
  G(\alpha, (\sigma,x)) = 
  \left(
  \begin{array}{c}
    \mu_0 \sigma + v_0^t x \\[1ex]
    F (\lambda_0 + \alpha \mu_0 + \sigma, u_0 + \alpha v_0 + x)
  \end{array}
  \right) \; .
\end{equation}
The zeros of~$G$ as the parameter~$\alpha$ varies correspond to the
pseudo-arclength continuation solutions of~$F$ for a single continuation
box. The first component of the function~$G$ guarantees that the pair~$(\sigma,x)$ is 
orthogonal to the direction~$(\mu_0,v_0)$. As we will show in the next subsection,
one can apply the constructive implicit function theorem from~\cite{sander:wanner:16a}
directly to the extended function~$G$ and thereby perform rigorously validated
pseudo-arclength continuation.

Since we will need them later, we close this subsection by explicitly
stating the derivatives of~$G$ with respect to both the variables~$(\sigma,x)$
and with respect to the parameter~$\alpha$. These are respectively given by
\begin{eqnarray}
  D_{(\sigma,x)}G(\alpha,(\sigma,x))  & = & \label{defextG1} \\[1ex]
  & & \hspace*{-2.5cm}
  \left( \begin{array}{cc}
    \mu_0 & v_0^t \\[1ex]
    D_\lambda F (\lambda_0 + \alpha \mu_0 + \sigma, u_0 + \alpha v_0 + x) &
    D_u F (\lambda_0 + \alpha \mu_0 + \sigma, u_0 + \alpha v_0 + x)
  \end{array} \right) \; , \nonumber
\end{eqnarray}
as well as
\begin{eqnarray}
  D_{\alpha}G(\alpha,(\sigma,x))  & = & \label{defextG2} \\[1ex]
  & & \hspace*{-2.5cm}
  \left( \begin{array}{c}
    0 \\[1ex]
    D_\lambda F(\lambda_0 + \alpha \mu_0 + \sigma, u_0 + \alpha v_0 + x)
    \mu_0 +
    D_u F(\lambda_0 + \alpha \mu_0 + \sigma, u_0 + \alpha v_0 + x) v_0
  \end{array} \right) \; . \nonumber
\end{eqnarray}

\subsection{Pseudo-arclength validation theorem}

We are now in a position to start establishing assumptions under which we can validate
a branch in the zero set of $F$ using pseudo-arclength continuation. For this we need 
the following modified set of assumptions. For the purposes of this paper, we use the vector norm~$\| (\alpha,x) \| = 
\max\{ |\alpha|, \| x \|_U \}$ for all $(\alpha,x) \in \R \times U$, even
though this could easily be modified.

\begin{itemize}
\item[(P1)] We assume both
\begin{equation} \label{eqn:xi}
  \| F(\lambda_0,u_0)\|_U \le \rho
  \quad\mbox{ and }\quad
  \| D_\lambda F (\lambda_0,u_0)\mu_0 + D_uF(\lambda_0,u_0) v_0 \|_U \le \xi
  \; .
\end{equation}
\item[(P2)] Assume that there exists an explicit constant $K>0$ which is a bound
on the operator norm of the inverse of the matrix
\begin{displaymath}
 D_{(\sigma,x)}G (0,(0,0)) = \left( \begin{array}{cc}
    \mu_0 & v_0^t \\[1ex]
    D_\lambda F(\lambda_0,u_0) & D_u F(\lambda_0,u_0)
  \end{array} \right) \; ,
\end{displaymath}
i.e., we suppose that
\begin{displaymath}
 \left\| D_{(\sigma,x)}G (0,(0,0))^{-1} \right\|_{\cL(\R \times U, \R \times U)}
 \le K \; .
\end{displaymath}
For this, we interpret the matrix as a linear map on the product space
$\R \times U$, and the operator norm is the norm in~$\cL(\R \times U, \R \times U)$.
\item[(P3)] Let~$M_1$, $M_2$, $M_3$, and~$M_4$ be Lipschitz constants such
that for all pairs~$(\lambda,u)$ which satisfy $\| u  - u_0\|\le d_u$ and
$| \lambda - \lambda_0| \le d_\lambda$ we have the estimates
\begin{eqnarray*}
  \| D_u F(\lambda,u) - D_u F(\lambda_0,u_0) \|_{\cL(U,U)} & \le &
    M_1 \| u-u_0 \|_{U} + M_2 |\lambda -\lambda_0 | \; , \\[1ex]
  \| D_\lambda F(\lambda,u) - D_\lambda F(\lambda_0,u_0)\|_{\cL(\R,U)} & \le &
    M_3 \| u-u_0 \| _U + M_4 | \lambda -\lambda_0 | \; ,
\end{eqnarray*}
where as usual we will identify the norm in~$\cL(\R,U)$ with the
norm~$\| \cdot \|_U$ in the following.
\end{itemize}
We would like to point out that all of the above three conditions are 
formulated in terms of the nonlinear parameter-dependent function~$F$ and
an approximate solution~$(\lambda_0,u_0)$ of the equation $F(\lambda,u) = 0$.

We now turn our attention to the extended system described by the operator~$G$
introduced in~(\ref{defextG}). It turns out that the above three assumptions
are tailor-made to establish the hypotheses~(H1) through~(H4) from the 
constructive implicit function theorem for the mapping~$G$. One can easily
see that~(P1) implies
\begin{displaymath}
  \| G(0,(0,0)) \|_{\R \times U} \le \rho \; ,
\end{displaymath}
i.e., hypothesis~(H1) is satisfied. Furthermore, using the explicit derivative
formulas from the end of the last subsection, the assumption~(P2) immediately
yields the estimate
\begin{displaymath}
  \| D_{(\sigma,x)}G(0,(0,0)) \|_{\cL(\R \times U,\R \times U)} \le K \; ,
\end{displaymath}
which establishes~(H2). It remains to show that~(P3) furnishes the estimates
in~(H3) and~(H4). For this, let~$\xi$ be defined as in~(\ref{eqn:xi}), and define
the four constants
\begin{eqnarray*}
  L_1 &=& \max(M_1+M_3, M_2+M_4) \; ,\\[1ex]
  L_2 &=& (M_1 + M_3) \| v_0 \|_{U} + (M_2 + M_4) |\mu_0| \; ,\\[1ex]
  L_3 &=& \xi \; , \\[1ex]
  L_4 &=& (M_1\| v_0\|_U + M_2 |\mu_0| ) \| v_0 \|_U +
    (M_3 \|v_0 \|_U + M_4 |\mu_0|)  |\mu_0| \; .
\end{eqnarray*}
Then the constants~$L_1$ through~$L_4$ are the Lipschitz constants for the
extended function~$G$ as required by~(H3) and~(H4). For this, first note that
in view of~(\ref{defextG1}) we have
\begin{displaymath}
  D_{(\sigma,x)}G(\alpha,(\sigma,x)) - D_{(\sigma,x)}G(0,(0,0)) =
  \left( \begin{array}{cc}
    0 & 0 \\[1ex]
    D_\lambda F (w_1) - D_\lambda F (w_2) &
    D_u F (w_1) - D_u F (w_2)
  \end{array} \right),
\end{displaymath}
where~$D_\lambda F$ and~$D_u F$ are evaluated at $w_1 = (\lambda_0 + \alpha \mu_0
+ \sigma, u_0 + \alpha v_0 + x)$ and $w_2 = (\lambda_0,u_0)$. Then one can readily
see that~(H3) follows from~(P3) and the estimates
\begin{eqnarray*}
  & & \|D_{(\sigma,x)}G(\alpha,(\sigma,x)) - 
    D_{(\sigma,x)}G(0,(0,0)) \|_{\cL(\R \times U, \R \times U)} \\[1ex]
  & & \qquad\quad \le \;
    \| D_u F(\lambda_0 + \alpha \mu_0 + \sigma, u_0 + \alpha v_0 + x) -
       D_u F(\lambda_0, u_0)\|_{\cL(U, U)} \\[1ex]
  & & \qquad\qquad\;
    + \; \| D_\lambda F(\lambda_0 + \alpha \mu_0 + \sigma, u_0 + \alpha v_0 + x) -
         D_\lambda F(\lambda_0, u_0)\|_{\cL(\R, U)}  \\[1ex]
  & & \qquad\quad \le \;
    M_1 ( |\alpha| \| v_0 \|_U + \| x \|_U) + M_2 ( |\alpha| |\mu_0| +
    |\sigma|) \\[1ex]
  & & \qquad\qquad\;
    + \; M_3 ( |\alpha| \| v_0 \|_U + \| x \|_U) + M_4 ( |\alpha| |\mu_0| +
    |\sigma|) \\[1ex]
  & & \qquad\quad = \;
    (M_1 + M_3) \|x \|_U + (M_2 + M_4) |\sigma | + ((M_1 + M_3)
    \| v_0 \|_U + (M_2 + M_4) |\mu_0| ) |\alpha| \\[1ex]
  & & \qquad\quad = \;
    L_1 \|(\sigma,x)\|_{\R \times U} + L_2 |\alpha| \; . 
\end{eqnarray*}
Similarly, using~(\ref{defextG2}) one can show that~(H4) follows from~(P1)
and~(P3), in combination with the inequalities
\begin{eqnarray*}
  & & \| D_{\alpha}G(\alpha,(0,0)) \|_{\cL(\R, \R \times U)} \\[1ex]
  & & \qquad\quad \le \;
    \| D_\lambda F(\lambda_0,u_0) \mu_0 + D_u F(\lambda_0,u_0)v_0 \|_U \\[1ex]
  & & \qquad\qquad\;
    + \; \|D_u F(\lambda_0 + \alpha \mu_0, u_0 + \alpha v_0)v_0 -
    D_u F(\lambda_0, u_0)v_0 \|_U \\[1ex]
  & & \qquad\qquad\;
    + \; \|  D_\lambda F(\lambda_0 + \alpha \mu_0, u_0 + \alpha v_0) \mu_0 -
    D_\lambda F (\lambda_0, u_0) \mu_0\|_U  \\[1ex]
  & & \qquad\quad \le \;
    \xi +  (M_1\| v_0\|_U + M_2 |\mu_0| ) |\alpha| \| v_0 \|_U +
    (M_3 \|v_0 \|_U + M_4 |\mu_0|)  |\alpha| |\mu_0| \\[1ex]
  & & \qquad\quad = \; L_3 + L_4 |\alpha| \; .  
\end{eqnarray*}
Altogether, these estimates lead to the following result.
\begin{theorem}[Pseudo-arclength continuation for a branch segment]
Consider the fixed pairs~$(u_0,\lambda_0)$ and~$(v_0,\mu_0)$ in~$\R \times U$,
let~$d_\lambda$ and~$d_u$ be two positive constants, and suppose that our
hypotheses~(P1), (P2), and~(P3) are satisfied. Moreover, assume that both
\begin{displaymath}
   4 K^2 \rho < 1
   \qquad\mbox{ and }\qquad
   2 K \rho < d_u
\end{displaymath}
hold. Then we can choose constants
\begin{displaymath}
  0 < \delta_\alpha \le d_\lambda  \; , \quad
  0 < \delta_u \le d_{u} \; , \qquad
  \mbox{ where }\qquad
  \delta_\alpha \|(\mu_0,v_0)\| + \delta_u \le \min(d_u,d_\lambda) \; ,
\end{displaymath}
and such that
\begin{displaymath}
  2 K L_1 \delta_u + 2 K L_2 \delta_\alpha \le 1
  \quad\mbox{ and }\quad
  2 K \rho + 2 K L_3 \delta_\alpha + 2 K L_4 \delta_\alpha^2 \le \delta_u
    \; .
\end{displaymath}
Then for every $\alpha \le \delta_\alpha$ there exists a unique~$(\sigma,x)$
in the zero set of~$G$ with~$\|(\sigma,x)\| \le \delta_u$. 

These statements guarantee that there is a unique element of the zero set
of~$F$ which lies on the hyperplane orthogonal to the center line in the slanted
box between~$(\lambda_0,u_0)$ and $(\lambda_0 + \delta_\alpha \mu_0, u_0+\delta_\alpha v_0)$
and passes through the point $(\lambda_0 + \alpha \mu_0, u_0+\alpha v_0)$.
This unique zero is given by $(\lambda_0 + \alpha \mu_0, u_0+\alpha v_0)+(\sigma,x)$.
Additionally, let
\begin{displaymath}
  \delta_{\min} = 2 K \rho \; .
\end{displaymath}
Then for $\alpha=0$ we can guarantee that the resulting pair in the zero 
of~$G$ is accurate within~$\delta_{\min}$ of~$(\lambda_0,u_0)$, 
and this zero is unique within the set $\|(\sigma,x)\| \le
\min \{ (2 K L_1)^{-1}, d_u, d_\alpha \}$.
\end{theorem}
\begin{proof}
To show the theorem we follow the proof of~\cite[Theorem~5]{sander:wanner:16a}.
Aside from the changes in the Lipschitz constants which have already been derived
before the formulation of the theorem, the  only changes to the cited proof 
are due to the fact that for a fixed parameter of~$G$, the values of both the 
parameter~$\lambda$ and the phase space value $x$ of~$F$ can vary. Therefore, in order to guarantee that 
the Lipschitz estimates on~$F$ hold, we need to assure that for every
$\alpha \le \delta_\alpha$ and all~$\|(\sigma,x)\| \le \delta_u$ the norm 
$\| \alpha (\mu_0,v_0) +  (\sigma,x)\|$ is bounded by both~$d_u$ and~$d_\lambda$.
This immediately leads to the additional constraints in the formulation of
the theorem.
\end{proof}
\begin{figure}[tb] \centering
     \includegraphics[height=4cm]{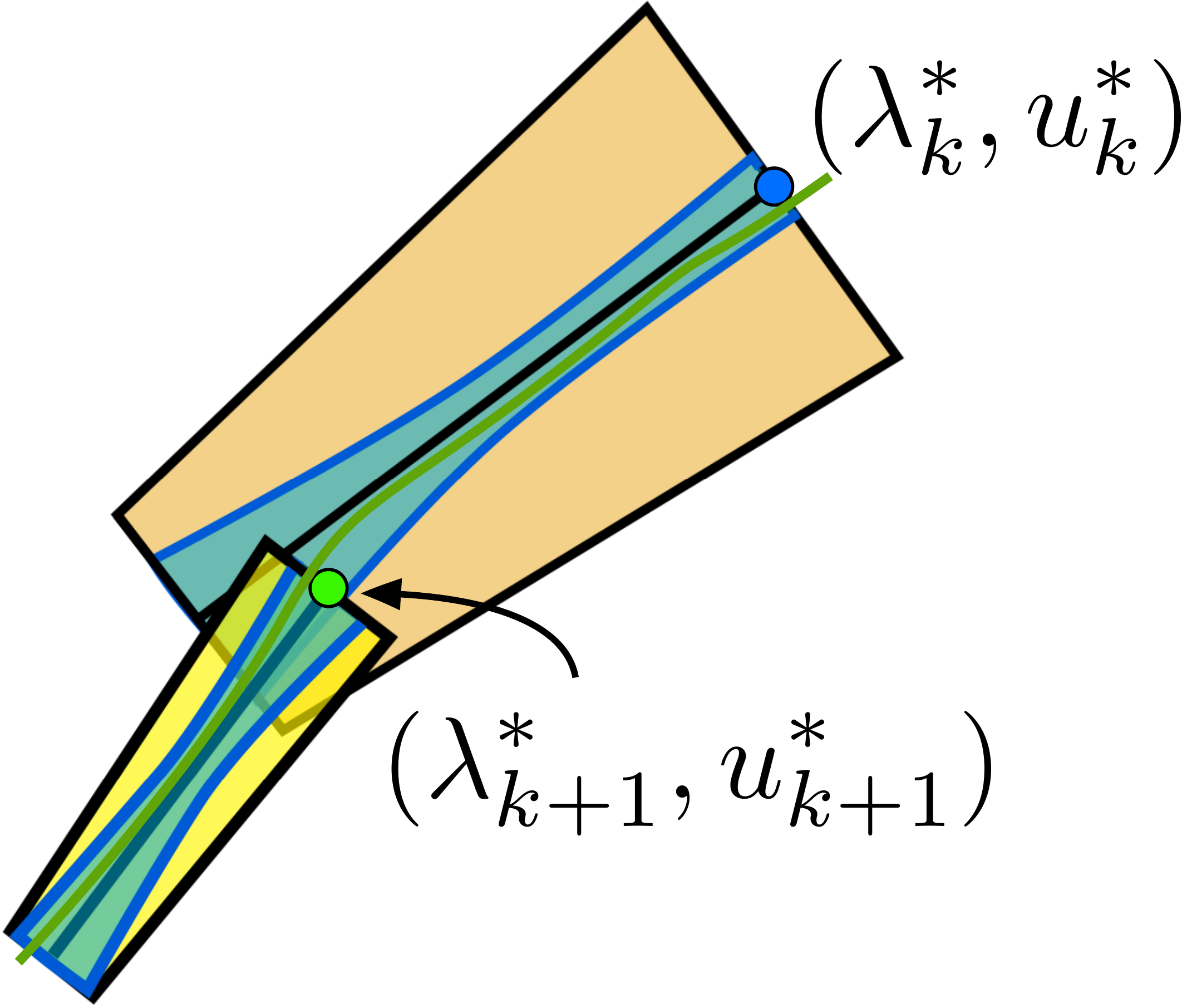}
     \hspace*{2cm}
     \includegraphics[height=4cm]{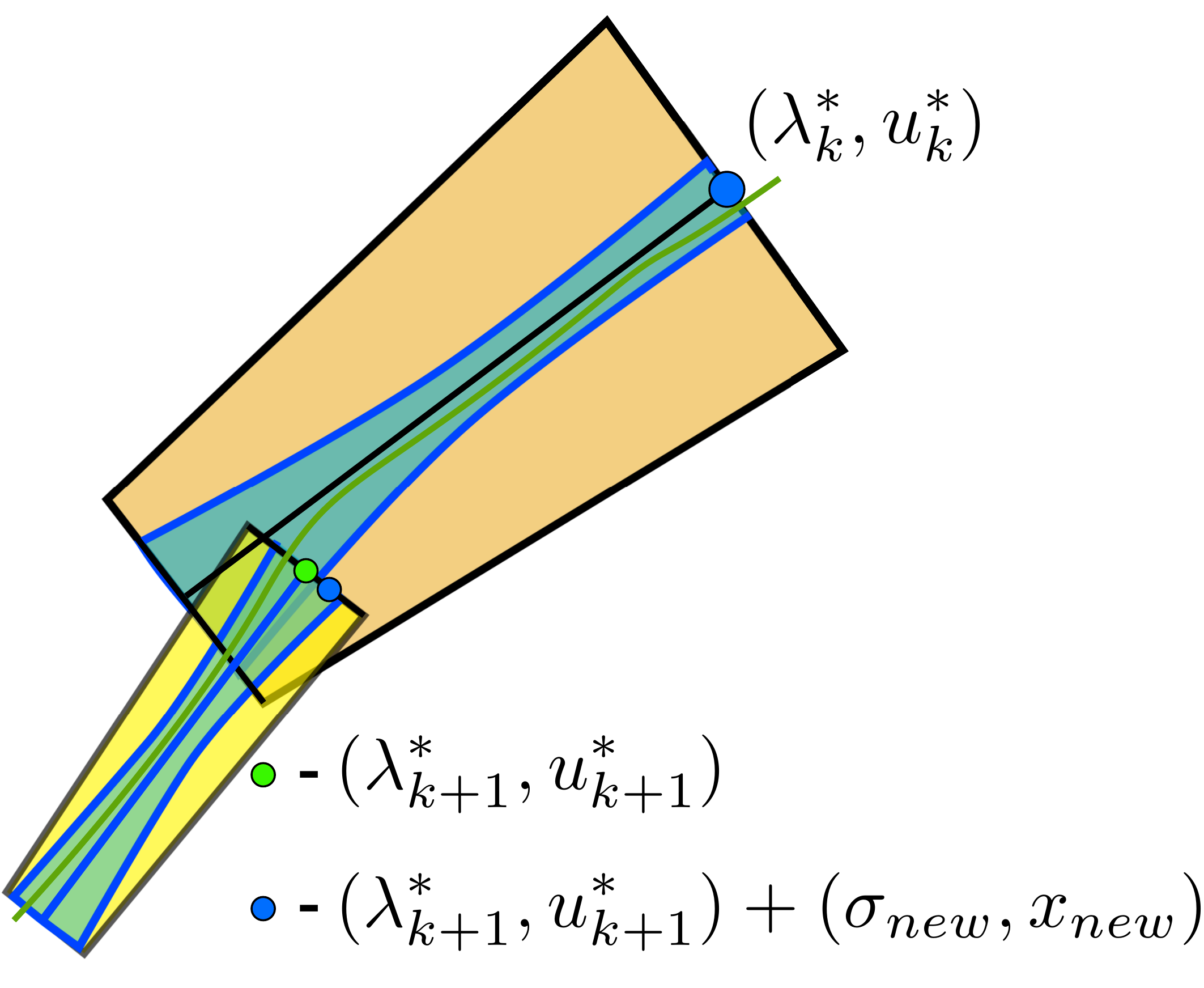}            
 \caption{Left image: Associated with each successive approximation, there is 
 a uniqueness region and an accuracy region. 
 Right image: In order to guarantee that the $k$-th and $(k+1)$-st region 
 enclose the same  component of the zero set  (the green curve), 
 we must verify the linking condition. This 
 requires that the accuracy curve of the $(k+1)$-st
 box at $\alpha=0$ (such as the blue point on the upper edge of the $(k+1)$-st blue box)  
 is contained in the uniqueness region of the $k$-th box (orange region). 
  }
  \label{fig:ac2}
\end{figure}%

The above theorem gives a  method for validating a branch segment of the zero set
within a single slanted box. In practice we use this result successively to validate
a whole solution branch. For each pair~$(\lambda_k^*, u_k^*)$, and for the approximate
tangent~$(\mu_k, v_k)$, we then define an extended function~$G_k$, and validate a 
branch segment for~$F$ within the k-th box. For a fixed parameter value $\alpha_k \le
\delta_\alpha$, we then use Newton's method to find an approximate zero of~$F$ which
is orthogonal to~$(\mu_k, v_k)$, i.e., which is a zero of~$G_k$. We abbreviate this
approximate zero as~$(\lambda_{k+1}^*, u_{k+1}^*)$, and can now repeat the entire
process for the $(k+1)$-st branch segment, see also Figure~\ref{fig:ac2}. What
remains to be shown is that the successive validated boxes are linked, meaning
that the branch segment in the $k$-th box and the branch segment in the $(k+1)$-st
box are on the same branch. That is, the accuracy region of the $(k+1)$-st box has
to be contained within the uniqueness region of the $k$-th box at the point~$\alpha_k$
where we made the numerical estimate. We give the linking condition for two boxes in
the next theorem.
\begin{theorem}[Linking branch segments]
Let $\delta_{k+1,\min} = 2 K_{k+1} \rho_{k+1}$ be the accuracy of the solution
\begin{displaymath}
  (\lambda_{k+1}^*,u_{k+1}^*) =
  (\lambda_k^* + \alpha_k \mu_k + \sigma^*, u_k^* + \alpha_k v_k + x^*)
  \; .
\end{displaymath}
In order to guarantee that the two validated boxes are linked, we require
the estimates
\begin{displaymath}
   |\alpha_k| + \frac{\delta_{k+1,\min}}{\|(\mu_k,v_k)\|} < \delta_{k,\alpha}
   \qquad\mbox{ and }\qquad
   |(\sigma^*,x^*)| + \delta_{k+1,\min} <  \delta_{k,u} \; .
\end{displaymath}
\end{theorem}
\begin{proof}
The accuracy of the $(k+1)$-st solution at $\alpha = 0$ is given by~$\delta_{k+1,\min}$.
That is, there exists a unique exact solution to $F=0$ of the form
\begin{displaymath}
  (\tilde{\lambda},\tilde{u}) = (\lambda_{k+1}^*+\sigma_{new},u_{k+1}^*+x_{new}) \; ,
\end{displaymath}
where $\|(\sigma_{new},x_{new})\| < \delta_{k+1,\min}$. In order to derive our linking
condition we need to establish that this solution is contained in the uniqueness region
of the $k$-th segment. We can therefore write
\begin{displaymath}
  (\tilde{\lambda},\tilde{u}) - (\lambda_k^*,u_k^*) =
  (\alpha_k + \alpha^+) (\mu_k,v_k) + (\sigma^* + \sigma^+,x^* + x^+) \; ,
\end{displaymath}
where $(\sigma_{new},x_{new}) = \alpha^+ (\mu_k,v_k) + (\sigma^+,x^+)$,
and~$(\mu_k,v_k)$ is orthogonal to~$(\sigma^+,x^+)$. Thus we have 
\begin{displaymath}
  \| \alpha^+ (\mu_k,v_k) +  (\sigma^+,x^+) \| < \delta_{k+1,\min} \; .
\end{displaymath}
By the orthogonality of the two vectors,  both the estimate $|\alpha^+|
\|(\mu_k,v_k)\| < \delta_{k+1,\min}$ and the estimate $\|(\sigma^+,x^+)\| <
\delta_{k+1,\min}$ are satisfied. In order to satisfy the linking condition,
we have to require that both $|\alpha_k + \alpha^+| < \delta_{k,\alpha}$
and $\|(\sigma^*+\sigma^+,x^*+x^+)\| < \delta_{k,u}$ hold.
This translates into the conditions
\begin{displaymath}
  |\alpha_k + \alpha^+| \le |\alpha_k| +
  \frac{\delta_{k+1,\min}}{\|(\mu_k,v_k)\|} < \delta_{k,\alpha} \; ,
\end{displaymath}
as well as
\begin{displaymath}
  \|(\sigma^*+\sigma^+,x^*+x^+)\| \le
  \|(\sigma^*,x^*)\| + \delta_{k+1,\min} < \delta_{k,u} \; . 
\end{displaymath}
This completes the proof of the theorem.
\end{proof}

\subsection{Preconditioning the coral map}\label{sec:preconditioning}

If we use the above method on the coral system, it is extremely slow to produce
the bifurcation diagram. This is due to the different relative sizes of the
components of the population and the parameter. We are able to significantly
speed up the method by using preconditioning. In particular, for $k = 1,\dots,d$
let
\begin{displaymath}
  \tilde{f}_k(\tilde{R},\tilde{u}) =
  \frac{f_k(100 \tilde{R}, (s_1 \tilde{u}_1, \dots, s_d \tilde{u}_d)) }{s_k}
  \; ,
\end{displaymath}
where~$s_1,\ldots,s_d$ are empirically determined positive scale constants. Then
it is clear that if we write $(R,u)= (100 \tilde{R}, (s_1 \tilde{u}_1,\dots, s_d \tilde{u}_d))$,
then~$(R,u)$ is a fixed point of~$f$ if and only if~$(\tilde{R},\tilde{u})$ is a
fixed point of the preconditioned map~$\tilde{f}$. However, the map~$\tilde{f}$ is
better scaled in the sense that we expect all components and the parameter to be of
the same order of magnitude. Therefore the pseudo-arclength continuation can be
performed more efficiently. In particular, we find that the size of~$\delta_\alpha$
in the preconditioned version is (in comparable coordinates) around two orders of
magnitude larger than those for the system without modification. This means that
we are able to validate a much larger portion of the bifurcation diagram with the
same number of continuation steps.  Figure~\ref{fig:comp}
shows 5000 continuation steps for the preconditioned case starting at 
$R = 300$ in the upper right corner, shown in blue. For comparison purposes, 
4000 continuation steps are shown  in red for the unmodified
case. The bifurcation curve goes through a limit point and almost to $\|u\|=0$
for the preconditioned case, but is hardly even a visible piece of red curve for
the original unmodified map. A similar preconditioning is performed in the case
of the bifurcation points, as described in the next section. 

\section{Validation of the bifurcation points}
\label{sec:validation}

In this section, we discuss the validation of the bifurcation points. Namely, 
we have used a computer-assisted proof to validate the Neimark-Sacker bifurcation point, 
where the invariant circles form in Section~\ref{subsec:hopf} and 
the saddle-node bifurcation point in Section~\ref{subsec:saddle}. 
In each case, to do so  we create an extended system $H$ such that $H=0$ guarantees
 the needed conditions for a bifurcation point. We then  apply
the constructive implicit function theorem to $H$. In both  cases, we use  interval arithmetic 
for a separate 
computational validation of the extra transversality and nondegeneracy conditions. 
We also prove that there is a  transcritical bifurcation point on the extinction axis. 
However, this last case does not require a computer-assisted proof for 
validation, since the calculations are simple enough for a closed form calculation.

\subsection{Validation of the Neimark-Sacker bifurcation point}
\label{subsec:hopf}

In Sections~\ref{sec:diagram} and~\ref{subsec:rotation}, we observed that at $(R,P) \approx (154.1,2689)$, there is 
a change in stability of the fixed points, and for $R>154.1$, typical initial conditions 
converge to populations which are  oscillating in time. This is the behavior associated with a Neimark-Sacker
bifurcation. 
In this section we detail the process of rigorous validation of the Neimark-Sacker 
bifurcation point seen in the upper right corner of Figure \ref{fig:bif}. While this
is the first time that a rigorous validation of a Neimark-Sacker bifurcation has been
performed in this way, rigorous validation of Hopf bifurcations was performed
in~\cite{vandenberg:etal:p20a} in the context of ordinary and partial differential equations,
but using a quite different method. Rather than considering conditions along a curve
of fixed points or equilibria, instead the method used a validated continuation of
periodic orbits with a renormalization technique, validating that there was a
bifurcation of equilibria at the turning point of this invariant closed curve
of solutions. Moreover, computer-assisted proofs were used in~\cite{capinski:etal:20a}
to rigorously establish an invariant circle in a two-dimensional map, which is created
via a Neimark-Sacker bifurcation. They do not, however, establish the bifurcation
point itself directly. While it would be interesting to adapt their method to the 
coral model, this lies beyond the scope of the current paper. 

We now proceed with our validation of 
the Niemark-Sacker bifurcation. As a first step, we state the standard theoretical Neimark-Sacker bifurcation 
theorem found in a bifurcation theory textbook. We then show how to adapt this classical result to create 
a rigorous computer-assisted  bifurcation theorem. 
\begin{theorem}[Neimark-Sacker bifurcation point] \label{thm:neimarksacker}
There is a Neimark-Sacker bifurcation for the coral system in~(\ref{eqn:model1})
and~(\ref{eqn:model2}) for the basic reproduction number $R_* \approx 154.1$ and
with polyp population density $P_* \approx 2689.$ The precise error bounds are stated in Table~\ref{tab:hopfproof}.
\end{theorem}
The remainder of this subsection is devoted to the proof of this theorem. Our approach
is to verify the classical conditions for a {\em Neimark-Sacker bifurcation\/}, as described
for example in~\cite{kuznetsov:98a} --- and which we briefly review in the following.
Consider a smooth map $f: \R \times \R^d  \to \R^d$. Furthermore, we begin by assuming
the following two conditions:
\begin{enumerate}
\item[(a)] {\em Existence of a fixed point:\/} The map~$f$ has a fixed point at a specific
parameter value, i.e., we assume that $f(\lambda_0,x_0) = x_0$.
\item[(b)] {\em Pair of imaginary eigenvalues on the unit circle:\/} The Jacobian
matrix $D_x f(\lambda_0,x_0)$ has exactly one simple conjugate pair of imaginary eigenvalues
on the unit circle. We denote these eigenvalues by~$e^{\pm i \theta_0}$, for some
angle~$0 < \theta_0 < \pi$.
\end{enumerate}
These two conditions have to be supplemented by another three transversality and
nondegeneracy conditions, which will be stated in detail below. For this, however,
we first need to introduce some additional notation.

Due to the implicit function theorem, as long as the Jacobian matrix in~(b) does not
have the eigenvalue~$1$, there exists a smooth curve of locally unique fixed points,
which we denote by $(\lambda,x_0(\lambda))$. Moreover, we define
\begin{displaymath}
  A(\lambda) = D_x f (\lambda,x_0(\lambda)) \; .
\end{displaymath}
We would like to point out that in our application to the coral system, the
rigorously established existence of the branch of fixed points as a side effect
also implies that along the branch near the Neimark-Sacker point, the Jacobian
matrix never has an eigenvalue~$1$.

Now let~$p \in \C^d$ and~$q \in \C^d$ denote the right eigenvectors of~$A(\lambda_0)$
corresponding to~$e^{i \theta_0}$ and~$e^{-i \theta_0}$, respectively, and normalized
in such a way that $\langle p,q \rangle = 1$, where the bracket notation denotes the
usual complex scalar product $\langle p,q \rangle := \overline{p}^t q$. Finally, by
Taylor's formula we can expand the function~$f$ in the form
\begin{equation} \label{eqn:ABCexpansion}
  f(\lambda_0,x) - x_0 =
  A(\lambda_0) x + \frac{1}{2} B(x,x) + \frac{1}{6} C(x,x,x) + O(\|x\|^4) \; ,
\end{equation}
where~$B$ and~$C$ denote the second- and third-order derivative terms at the
point~$(\lambda_0,x_0)$ in the form
\begin{displaymath}
  B_i(y,z) = \sum_{j,k=1}^{d} \dfrac{\partial^2 f}{\partial x_j
    \partial x_k}(\lambda_0,x_0) y_j z_k
  \quad\mbox{ and }\quad
  C_i(y,z,w) = \sum_{j,k,l=1}^{d} \dfrac{\partial^3 f}{\partial x_j
    \partial x_k \partial x_l}(\lambda_0,x_0) y_j z_k w_l \; . 
\end{displaymath}
After these preparations, we can now complete our description of the conditions needed
for the Neimark-Sacker theorem:
\begin{enumerate}
\item[(c)] {\em Transversality condition:\/} Using the notation above, suppose that
\begin{displaymath}
  \mathrm{Re} \left( e^{-i \theta_0} \left\langle p, \frac{dA}{d\lambda}(\lambda_0)q
  \right\rangle \right) \ne 0 \; .
\end{displaymath}
\item[(d)] {\em Nondegeneracy condition I:\/} Suppose that
\begin{displaymath}
  \theta_0 \ne \frac{\pi}{2}
  \quad\mbox{ and }\quad
  \theta_0 \ne \frac{2 \pi}{3} \; .
\end{displaymath}
\item[(e)] {\em Nondegeneracy condition II:\/} Suppose that
\begin{eqnarray*}
  Re \left( \right.  e^{-i \theta_0} ( \langle p,C(q,q,\bar{q}) \rangle
   + 2 \langle p,B(q,(I-A)^{-1} B(q,\bar{q})) \rangle & & \\[1ex]
  + \langle p, B(\overline{q}, (e^{2 i \theta_0} I - A)^{-1} B(q,q)) \rangle )
  \left. \right) & \ne & 0 \; .
\end{eqnarray*}
\end{enumerate}
To summarize, the transversality condition implies that the pair of complex
conjugate eigenvalues at~$\lambda_0$ crosses the imaginary axis with nonzero
speed. The first nondegeneracy condition indicates that the eigenvalues~$e^{\pm i\theta_0}$
are not $k$-th roots of unity for $k = 1,\ldots,4$. Since the proof of the 
Neimark-Sacker theorem is based on the Poincar\'e normal form theorem, this
condition excludes resonances. Finally, the left-hand side of the second
nondegeneracy condition gives the coefficient of the cubic term in the
complex Poincar\'e normal form, and its sign distinguishes between a sub-
and super-critial Neimark-Sacker bifurcation. For more details we refer the
reader to the part of~\cite[Section~5.4]{kuznetsov:98a} devoted to the 
Neimark-Sacker bifurcation.

Under the above conditions, the Neimark-Sacker theorem guarantees that a locally
unique invariant closed curve bifurcates from the set of fixed points at the
point~$(\lambda_0,x_0)$. As already mentioned, the type of bifurcation depends
on the sign of the left-hand side of~(e).
%
\begin{table}[tb] \centering
  \begin{tabular}{|c|c|c|c|c|c|}
    \hline
   $R$ &  $\lambda$ & $x_1$ & $P$ &  $\delta_1$ & $\delta_2$ \\
    \hline
   $154.1$ & $5.286$ & $1794$ & $2689$ & $1.473 \cdot 10^{-10}$ & $1.220
     \cdot 10^{-8}$\\
   \hline\hline
   $\rho$ &  $K$ & $L_1$ & (c) & (d) & (e)  \\
     \hline    
     $6.166 \cdot 10^{-11}$ & $1.000$ & $4.097 \cdot 10^{7}$ & $4.338 \cdot 10^{-2}$
     & $46.85$ & $-1.21 \cdot 10^{-6}$\\
    \hline
  \end{tabular}
  \caption{Validation constants for the system~(\ref{eq:hopf42}) at the
  Neimark-Sacker bifurcation point. All values are written with four decimal
  places, unless less accuracy is known. For more efficient computation, we
  multiplied by a preconditioning matrix and determined the bounds~$\rho$, $K$,
  and~$L_1$. We selected a matrix close to the Jacobian matrix of~$H_{ns}$, whose inverse
  was used as a preconditioner. The accuracy constant~$\delta_1$ and the isolation
  bound~$\delta_2$ were derived using~$\rho$, $K$ and~$L_1$. For the three conditions~(c),
  (d), and~(e), which were checked separately after the validation involving~$H_{ns}$, we
  used an interval arithmetic enclosure of the approximate solution with radius~$\delta_1$.
  Note that the angle in~(d) is given in degrees.}
  \label{tab:hopfproof}
\end{table}%

In order to create the validation version of this theorem, we use a suitable
extended system to validate assumptions~(a) and~(b). After having established
an existence and uniqueness result for this extended system, one can then validate
conditions~(c), (d), and~(e) separately using interval arithmetic. For convenience, 
we have converted the complex system into the following real system of equations.  
We are seeking zeros of the function $H_{ns}: \R^m \to \R^m$, which is 
defined as
\begin{equation} \label{eq:hopf42}
  H_{ns}(x,\lambda,w,u,a,b) = \left(
  \begin{array}{c}
    f(\lambda,x) - x\\[0.5ex]
    D_xf(\lambda,x) w - aw + bu \\[0.5ex]
    D_xf(\lambda,x)u - bw - au \\[0.5ex]
    a^2 + b^2 - 1 \\[0.5ex]
    \Vert w \Vert^2 - 1 \\[0.5ex]
    \Vert u \Vert^2 - 1 
  \end{array}
  \right) \; .
\end{equation}
The first equation in the system is the fixed point condition. The second through
fourth equations form the simple complex eigenvalue pair condition, where we write
$e^{\pm i \theta_0} = a \pm i b$, and the eigenvectors~$p$ and~$q$ are given by
$u \pm i w$, up to normalization. The last two equations are included to single
out a locally unique eigenvector.

For a function of the form $f: \R \times \R^d \to \R^d$, we have $x \in \R^d$,
$\lambda \in \R$, $u,w \in \R^d$, as well as $a, b \in \R$. Therefore, the extended
system $H_{ns}: \R^m \to \R^m$ lives in dimension $m = 3 d + 3$. In our numerical
validation, we are working with a $13$-dimensional system, implying that this
extended system has dimension~$42$.

Using standard numerical methods, we obtained an approximate bifurcation point
satisfying~$H_{ns}(x,\lambda,w,u,a,b) = 0$, for the function~$H_{ns}$ in~(\ref{eq:hopf42}),
and with values for~$R$, $\lambda$, $x_1$, and~$P$ as stated in
Table~\ref{tab:hopfproof}. Since~$H_{ns}$ is parameter free, we only seek rigorous
solutions of the extended system in~(\ref{eq:hopf42}) which satisfy $H_{ns} = 0$
in~$\R^{42}$. Thus we only need to verify the hypotheses of the constructive
implicit function theorem which involve the values of~$\rho$, $K$, $L_1$,
and~$\ell_x > 0$ at our computed approximation point. See also Theorem~\ref{nift:thm}.
Table \ref{tab:hopfproof} summarizes  the constants found for the validation of the
solution of system (\ref{eq:hopf42}).

We obtain the bounds~$\rho$ and~$K$ by using interval arithmetic. While the bound~$\rho$
can be found in a straightforward way, the constant~$K$ cannot easily be found by using
interval arithmetic to compute matrix inverses. Therefore, we first compute an approximate
numerical inverse. However, we still need a bound on the exact inverse, and a bound
on the accuracy of the approximate inverse. This is required in both the computation
of~$K$ and twice when we verify condition~(e). The required quantities can be determined
using the following lemma. While we apply this lemma only for matrices, it is stated for
the case of Banach spaces. 
\begin{lemma}[Inverse bounds]
Let~$A$ be a bounded linear operator between two Banach spaces, and let~$B$ be an
approximate inverse of~$A$. Assume further that
\begin{displaymath}
  \| I  - BA \| \le \rho_1 < 1
  \qquad\mbox{ as well as }\qquad
  \|B\| \le \rho_2 \; .
\end{displaymath}
Then~$A$ is one-to-one, onto, and we have both
\begin{displaymath}
  \| A^{-1} \| \le \frac{\rho_2}{1-\rho_1}
  \qquad\mbox{ and }\qquad
  \| B- A^{-1} \| \le \frac{\rho_1 \rho_2}{1-\rho_1} \; .
\end{displaymath}
\end{lemma}
The bound on~$A^{-1}$ is due to a Neumann series argument, and the proof can
be found in~\cite{sander:wanner:16a}. In addition, the second bound is a
consequence of $\| B - A^{-1}\| \le \|I-BA\| \|A^{-1}\|$.

Having described how the constants~$\rho$ and~$K$ can be estimated rigorously,
we now turn our attention to the Lipschitz constant~$L_1$. It can be determined
using the mean value theorem for multivariate functions from the calculations
in~(\ref{eq:mult}) below. For this, suppose that the function $H_{ns}: \R^m \rightarrow \R^m$
is differentiable and let $h_{ij}(x) = (\partial (H_{ns})_i / \partial x_j)(x)$.
Then $h_{ij}: \R^m \rightarrow \R$, and we let $h : \R^m \rightarrow \R^{m \times m}$
denote the matrix-valued function with entries~$h_{ij}$. Throughout our computations,
we used the maximum norms for vectors~$x$, and the induced matrix norm for matrices~$A$.
Recall that one then has $\Vert x \Vert = \Vert x \Vert_\infty = \max_{i=1,\ldots,m}
\vert x_i \vert$, as well as $\Vert A \Vert = \Vert A \Vert_\infty = \max_{i=1,\ldots,m}
\sum_{j=1}^m \vert A_{ij} \vert$. After these preparations, the mean value theorem
implies
\begin{displaymath}
  |h_{ij}(x) - h_{ij}(y)| \leq \max_{c \in D} \Vert \nabla h_{ij}(c) \Vert_1
    \, \Vert x-y \Vert \; , 
\end{displaymath}
where~$D$ denotes the line segment between the points~$x$ and~$y$. Together with
the definition of the functions~$h_{ij}$ one further obtains
\begin{eqnarray*}
  |h_{ij}(x) - h_{ij}(y)| & \leq &
    \max_{c \in D} \left\lVert \left(\dfrac{\partial^2 (H_{ns})_i}{\partial x_1
    \partial x_j}(c), \ldots, \dfrac{\partial^2 (H_{ns})_i}{\partial x_n
    \partial x_j}(c) \right) \right\rVert_1 \, \Vert x - y \Vert \\[1ex]
  & \leq & m \max\limits_{c \in D, \, k=1,\ldots,m}
    \left\lvert \dfrac{\partial^2 (H_{ns})_i}{\partial x_k \partial x_j}(c)
    \right\rvert \, \Vert x-y \Vert \; .
\end{eqnarray*}
This finally furnishes
\begin{eqnarray}
  \Vert h(x) - h(y) \Vert & = &
    \max_{i=1,\ldots,m} \sum_{j=1}^m \vert h_{ij}(x) - h_{ij}(y) \vert
    \nonumber \\[1ex]
  & \leq & \max_{i=1,\ldots,m} \sum_{j=1}^m \left( m
    \max\limits_{c \in D, \, k=1,\ldots,m} \left\lvert
    \dfrac{\partial^2 (H_{ns})_i}{\partial x_k \partial x_j}(c) \right\rvert
    \right) \Vert x-y \Vert \, . \label{eq:mult}
\end{eqnarray}
The factor in front of~$\Vert x-y \Vert$ on the right-hand side is then the
Lipschitz constant~$L_1$, and it can be determined via interval arithmetic 
and automatic differentiation.

Altogether, our rigorous computer-assisted proof of Theorem~\ref{thm:neimarksacker}
can be summarized as follows. After completing the validation of the conditions 
that guarantee that the constructive implicit function theorem holds, we are able
to verify the accuracy and uniqueness regions for the bifurcation point. In addition,
we can use Intlab~\cite{rump:99a} to rigorously show that the Jacobian matrix
$D_xf(\lambda_0,u_0)$ has in fact only two eigenvalues on the unit circle, by 
verifying that the remaining eleven eigenvalues all lie inside the unit disk.
This implies that a bifurcation occurs within the specified error of the
approximate bifurcation point. We then verify that this bifurcation is indeed 
a Neimark-Sacker bifurcation by showing that conditions~(c), (d), and~(e) hold
using interval arithmetic on these conditions. Here are a few remarks which give
a more detailed explanation:
\begin{itemize}
\item For each condition, we show that the interval containing the exact answer
does not contain zero for~(c) and~(e), and does not contain any of the avoided
angles for~(d).
\item While we are able to work with real-valued quantities $a,b,u,v$ in the
initial calculations of parts (a) and (b), we must switch to the complex case to verify the extra
conditions (c), (d), (e), and we normalize the complex vectors~$p$ and~$q$ using the normalization
condition $\langle p,q \rangle = 1$.
\item We need to be able to guarantee that all three conditions are satisfied for the
entire accuracy region. Therefore we evaluate these conditions on an interval vector
whose midpoint is the approximate bifurcation point, and whose radius is~$\delta_1$.
That is, every component of the vector is an interval. The actual computed values of
the conditions (c)-(e) are intervals, but the values given in Table~\ref{tab:hopfproof}
are the worst-case scenario values. Even with the interval calculations, conditions~(c)
and~(d) are known to more than four significant digits, but condition~(e) is only
known to three digits of accuracy.
\end{itemize}
This completes the proof of Theorem~\ref{thm:neimarksacker}.


\subsection{Validation of the saddle-node bifurcation point}
\label{subsec:saddle}

In this section, we use a computer-assisted proof to show that there is a
saddle-node bifurcation point in the coral model. The precise result can
be stated as follows.
\begin{theorem}[Saddle-node bifurcation point] \label{thm:saddle-node}
The coral model in~(\ref{eqn:model1}) and~(\ref{eqn:model2}) has a saddle-node
bifurcation point near the basic reproduction number $R_* \approx 12.28$, which
corresponds to the parameter value $\lambda_* \approx 0.4213$, and for polyp
population density $P_* \approx 853.4$. The precise error bounds are stated in 
Table~\ref{tab:snproof}. 
\end{theorem}
\begin{table}[tb] \centering
  \begin{tabular}{|c|c|c|c|c|c|}
    \hline
   $R$ & $\lambda$ & $x_1$ & $P$ & $\delta_1$ & $\delta_2$ \\
    \hline
   $12.28$ & $0.4213$ & $569.5$ & $853.4$ & $3.306 \cdot 10^{-12}$ &
     $4.015 \cdot 10^{-7}$\\
        \hline\hline
      $\rho$ & $K$ & $L_1$ & (c) & (d) & \\
      \hline    
     $1.653 \cdot 10^{-12}$& $1$ & $1.245 \cdot 10^{6}$ & $-353.4$ & $-9.924 \cdot 10^{-4}$&\\
    \hline
  \end{tabular}
  \caption{Validation constants for the extended system in~(\ref{eqn:sn}) at the
  saddle-node bifurcation point. All values are written up to four decimal places.
  For more efficient computation, we multiplied by a preconditioning matrix and
  obtained the bounds~$\rho$, $K$, and~$L_1$. We selected a matrix close to the
  Jacobian matrix of~$H_{sn}$, whose inverse was used as a preconditioner. The accuracy
  constant~$\delta_1$ and the isolation bound~$\delta_2$ were derived using~$\rho$,
  $K$, and~$L_1$. For the two conditions~(c) and~(d), which were checked separately
  after the validation involving~$H_{sn}$, we used an interval arithmetic enclosure of
  the approximate solution with radius~$\delta_1$.}
  \label{tab:snproof}
\end{table}%
As in the previous subsection, the remainder of the present one is devoted to the
verification of this theorem via computer-assisted rigorous methods. In order to
establish the theorem, we need to verify the following conditions from the
{\em classical saddle-node bifurcation theorem\/}, see for example~\cite{kuznetsov:98a}.
Let $f: \R \times \R^d \to \R^d$ be a smooth mapping. Furthermore, assume the following
four conditions:
\begin{enumerate}
\item[(a)] {\em Existence of a fixed point:\/} The map~$f$ has a fixed point at a specific
parameter value, i.e., we assume that $f(\lambda_0,x_0) = x_0$.
\item[(b)] {\em Simple eigenvalue~1:\/} The Jacobian matrix $D_x f(\lambda_0,x_0)$
has a simple eigenvalue of~$1$. Let~$p$ and~$q$ denote the corresponding left and
right eigenvectors, and suppose they are normalized to satisfy $p^t q = 1$.
\item[(c)] {\em Transversality condition:\/} Using the above notation we assume
\begin{displaymath}
  p^t D_{\lambda}f(\lambda_0,x_0) \ne 0 \; .
\end{displaymath}
\item[(d)] {\em Nondegeneracy condition:\/} Now let $A(\lambda_0) = D_xf(\lambda_0,x_0)$,
and consider the expansion of~$f$ given in~(\ref{eqn:ABCexpansion}). Then we suppose
further that
\begin{displaymath}
  p^t B(q,q) \ne 0 \; .
\end{displaymath}
\end{enumerate}
Then the classical saddle-node bifurcation theorem guarantees a saddle-node
bifurcation at the pair~$(\lambda_0,x_0)$.

In order to validate our bifurcation point using this theorem, we use again an
extended system of the form~$H_{sn} = 0$ to validate conditions~(a) and~(b), and
then we verify conditions~(c) and~(d) separately afterwards. This time,
the extended mapping~$H_{sn}$ is a map $H_{sn}: \R^{27} \to \R^{27}$, and it is
defined as
\begin{equation}\label{eqn:sn}
  H_{sn}(x,v,\lambda) = \left(
  \begin{array}{c}
    f(\lambda,x) - x \\[0.5ex]
    D_x f(\lambda,x) v - v \\[0.5ex]
    \|v\|^2 - 1
  \end{array} \right) \; .
\end{equation}
In order to validate~(c), and~(d), we use interval arithmetic for both of
these conditions, and show that~$0$ does not lie in the interval containing
the resulting answer. Note that the vector~$q$ is just a multiple of~$v$,
and~$p$ can be found in a verified way using Intlab~\cite{rump:99a}. The
summary of the constants of this validation process is given in
Table~\ref{tab:snproof}. This computer-assisted proof is quite similar
to the one used for the Neimark-Sacker bifurcation in the last subsection, 
and therefore we do not give any more elaboration on the technique used to
compute these values. This completes the proof of Theorem~\ref{thm:saddle-node}.


\subsection{Validation of the transcritical bifurcation point}
\label{subsec:transcritical}

We close this section by showing that there is indeed a transcritical
bifurcation on the trivial solution curve, i.e., the extinction curve.
This time, it is not necessary to perform a computer-assisted proof, as
the bifurcation can be established directly by hand.
\begin{theorem} [Transcritical bifurcation point]

For the coral population model in~(\ref{eqn:model1}) and~(\ref{eqn:model2})
there exists a transcritical bifurcation point for basic reproduction
number $R_* = c_2/c_1 \approx 72.22$, which corresponds to the parameter
value $\lambda_* = R_*/(b \cdot a)$ and to $x_* = 0 \in \R^{13}$.
Recall that the constants~$c_1$ and~$c_2$ were introduced in~(\ref{eqn:phi}),
and the vectors~$a$ and~$b$ were defined in~(\ref{def:a}) and the following 
paragraph.
\end{theorem}
\begin{proof}
It is clear from the model that $x = 0$ is a fixed point for all values of
the parameter~$\lambda$. Furthermore, one can easily show that
\begin{displaymath}
  \det (D_xf(\lambda,0) - I) =
  \lambda - \frac{c_2}{c_1 (b \cdot a)} \; .
\end{displaymath}
Therefore, the Jacobian matrix of~$f(\lambda,\cdot)$ at the origin has a
simple eigenvalue of~1 if and only if~$\lambda$ equals
\begin{displaymath}
  \lambda_* = \frac{c_2}{c_1 (b \cdot a)} \; .
\end{displaymath}
Now denote the right and left eigenvectors of~$D_xf(\lambda_*,x_*)$ by~$v$
and~$w$, respectively. One can show directly that $v = a$ defined in~(\ref{def:a}), and 
$w$ is such that 
\begin{displaymath}
  w_1 = b \cdot a\; , \quad
  w_{d} = b_{d} \; ,
  \quad\mbox{ and }\quad
  w_k = b_k + S_k w_{k+1}
  \quad\mbox{ for }\quad
  k=2,\dots,d-1 \; .
\end{displaymath}
Then in order
to establish the transcritical bifurcation, two nondegeneracy conditions have to 
be verified. Since we have $w^t D_\lambda f(\lambda_*,x_*) = 0$, one first has
to show that
\begin{displaymath}
  w^t D_{x\lambda} f(\lambda_*,x_*)v = 
  \frac{(b \cdot a) c_1}{c_2} w^t b  
\end{displaymath}
is nonzero, which is clearly satisfied since all the terms of $b$ and $w$ are non-negative, 
and contains terms of the form $b_k^2$ (which are strictly positive for each nonzero $b_k$). 

Second, we need to show that
$w^t D_{xx}f(\lambda_*,x_*)[v,v] \ne 0$. Since only the first component of~$f$,
which we call~$f_1$, is nonlinear, one merely needs to consider the second
derivative of this component function. We get the following formula. 
\begin{displaymath}
  D_{xx}f_1(\lambda_*,x_*)[v,v] =   \frac{ 2 (\beta-\alpha)  }{\Omega} \sum_{k=2}^d p_k a_k . 
\end{displaymath}
By looking at the corresponding parameter values, this value is also nonzero, and
therefore the second nondegeneracy condition holds. This completes the proof of
the theorem. 
\end{proof}
%
%


\section{Conclusion}
\label{sec:con}

In this paper, we have considered an age-structured population model for red coral populations
with a parameter of fitness. When the fitness increases sufficiently, a set of stable invariant
closed curves of oscillating orbits form, and these stable curves persist for large values of
the fitness parameter. It is not surprising that for small fitness parameters, solutions limit
to extinction, but we see that even for large fitness, populations become extremely vulnerable,
as they limit to oscillation spending long period of time near extinction. 

The coral population model has a curve of fixed points containing a Neimark-Sacker, saddle-node,
and transcritical bifucation point. We develop new methods based on previous computer-assisted
proof methods  and use these methods to validate the branch of fixed points, and the three
bifurcation points. 

\section*{Acknowledgments}
We would like to thank Konstantin Mischaikow for pointing us to this coral population model. 
This research was partially supported by NSF grant DMS-1407087.
In addition, E.S.\ and T.W.\ were partially supported by the
Simons Foundation under Awards~636383 and~581334, respectively.

%
%

\addcontentsline{toc}{section}{References}
\footnotesize
%
%
\bibliography{wanner1a,wanner1b,wanner2a,wanner2b,wanner2c,coralbib}
\bibliographystyle{abbrv}
\end{document}